\documentclass[leqno,11pt]{amsart}
\usepackage{amssymb, amsmath}
\usepackage[usenames,dvipsnames]{xcolor}
\usepackage{color,ulem,textcomp}

\numberwithin{equation}{section}
\usepackage{enumerate}
\usepackage{graphicx}
\usepackage{cancel}

\theoremstyle{definition}
\theoremstyle{plain}
\newtheorem{thm}{Theorem}[section]
\newtheorem{Prop}[thm]{Proposition}

 \newtheorem{Thm}{Theorem}[section]
 \newtheorem{Rmk}[thm]{Remark}
 \newtheorem{Lem}[thm]{Lemma}
  \newtheorem{Def}[thm]{Definition}

 \def\N {\mathbb{N}}
\def\R {\mathbb{R}}

\def\T {\mathbb{T}}

\def\cA {\mathcal{A}}
\def\cB {\mathcal{B}}
\def\cC {\mathcal{C}}

\def\cD {\mathcal{D}}

\def\cG {\mathcal{G}}

\def\cN {\mathcal{N}}

\def\cZ {\mathcal{Z}}

\def\cR {\mathcal{R}}
\def\cT {\mathcal{T}}

\def\b {{\beta}}

\def\eps {{\varepsilon}}
\def\e {{\varepsilon}}

\def\indc {{\bf 1}}

\def\d {{\partial}}

\newcommand{\ba}{\begin{aligned}}
\newcommand{\ea}{\end{aligned}}

\newcommand{\be}{\begin{equation}}
\newcommand{\ee}{\end{equation}}

\numberwithin{equation}{section}

\begin{document}


\title[One-sided convergence and irreversibility ] 
{One-sided convergence\\ in the Boltzmann-Grad limit}
\author{Thierry Bodineau, Isabelle Gallagher,\\ Laure Saint-Raymond and Sergio Simonella}

\begin{abstract}
We review various contributions   on the fundamental work of Lanford \cite{lanford}
deriving the Boltzmann equation from   hard-sphere dynamics in the low density limit.

We focus especially on the assumptions made on the initial data and on how they encode irreversibility.
The impossibility to reverse time in the Boltzmann  equation
(expressed for instance by Boltzmann's H-theorem) is related to the lack of convergence of higher order marginals on some singular sets.
Explicit counterexamples single out the microscopic sets where the initial data should converge in order to produce the Boltzmann dynamics.

\end{abstract}
\null

\maketitle

\section{Introduction}

\subsection{Goals}

The Boltzmann equation was introduced at the end of the nineteenth century to predict the statistical behavior of a perfect gas out of thermodynamic equilibrium.
This equation expresses the transport and   collisions of microscopic particles (atoms)  which are supposed to interact typically as elastic hard spheres.

However the resulting dynamics exhibits very different features compared to the {\color{black} reversible} deterministic system of hard spheres, which is   a Hamiltonian system.
 The Boltzmann equation  generates indeed a semi-group with a Lyapunov functional (the entropy increases along the evolution), and an attractor as time goes to infinity (the density converges to thermodynamic equilibrium).
 These discrepancies between {\color{black} the microscopic and the macroscopic descriptions} were the starting point of some very violent controversy opposing for instance Boltzmann to Loschmidt \cite{Cercignani,boltzmann,Bo96,loschmidt}. 
 There is still an important challenge in understanding the origin of the {\bf non-reversible} Boltzmann equation and the conditions under which it can provide a good approximation of the microscopic dynamics. 
{\color{black} 
We refer to \cite{Lebowitz Goldstein,Lebowitz 2} for a review on the irreversibility and on the key role played by entropy 
and to \cite{Sp97} for a modern perspective on Loschmidt's argument.
In this paper, we will focus on a more quantitative analysis of the mathematical aspects leading to the emergence of irreversibility.
}

\medskip
 The convergence result describing at best up to now this transition is due to Lanford~\cite{lanford}. It states that the Boltzmann equation can be obtained as the limit of the deterministic dynamics in a box of size 1
 \begin{itemize}
 \item in the low density regime, i.e. as the number of particles $N\to \infty$, their size $\eps \to 0$,   with the additional condition that the inverse mean free path $N\eps^{d-1}$ remains of order~1 (where~$d$ is the space dimension);
 \item up to excluding some pathological situations which occur with vanishing probability in this limit; \item provided that   initially the particles are distributed independently. \end{itemize}
One important restriction is that this convergence result holds only for short times, which 
{is not enough for}
observing any relaxation towards equilibrium. 
Despite many efforts, this restriction  has  not been removed to this day.
There is no attempt in the present paper to improve the convergence time. Our goal here  is to study the appearance of irreversibility which already occurs for short times.

 \bigskip
More precisely, we intend to discuss in detail the assumptions on the initial data in Lanford's theorem, as they encode all the information on the future evolution. The statement is the following.

\begin{Thm}[\cite{lanford}]
\label{lanford-thm}
Consider a system of $N$ hard-spheres of diameter $\eps$ on the~$d$-dimensional periodic box $\T^d = [0,1]^d$ (with~$d\geq 2$), initially ``independent" and identically distributed with continuous density 
$f_0$ such that 
\begin{equation}
\label{eq: initial data 1}
\big\| f_0 \exp(\mu + \frac{\beta}{2}|v|^2 ) \big\|_{L^\infty(\T^d_x \times \R^d_v)}\leq 1 \, ,
\end{equation}
 for some $\beta>0, \mu \in \R$.
For instance, we can choose  the initial distribution of~$N$ particles with minimal correlations,  due only  to the non overlapping conditions~:
 \begin{equation}
\label{eq: initial data 2}
f_{N,0} (x_1, v_1, \dots, x_N, v_N ) = \frac1{\cZ_N} \prod_{i=1}^Nf_0(x_i, v_i)  \prod_{i\neq j} \indc_{|x_i-x_j|>\eps }\,,
\end{equation}
 denoting by $\cZ_N$ the partition function, that is the normalizing constant for $f_{N,0}$ to be a probability.
 
In the Boltzmann-Grad limit~$N \to \infty$ with $N \eps^{d-1} = 1$, the one particle distribution $f_N^{(1)}=f_N^{(1)}(t,x,v)$   converges almost everywhere to the solution of the Boltzmann equation
\begin{equation}
\label{boltzmann}
\left\{ \begin{aligned}
& \d_t f +v\cdot \nabla_x f=   Q(f,f) \, ,\\
& Q(f,f)(v):=\iint_{{\mathbb S}^{d-1} \times \R^d} [f(v')f(v'_1)-f(v)f(v_1)]  \,
{\color{black} \big((v_1-v)  \cdot \nu\big)_+ } \, dv_1 d\nu  \, ,\\
& v'=v +\nu \cdot (v_1-v)  \, \nu \, , 
	\quad v'_1=v_1 - \nu \cdot(v_1-v) \, \nu \,  ,
 \end{aligned}\right.
\end{equation}
with initial data~$f_0$, {\color{black} on a time interval $[0, t^*]$ where  $t^*$ depends only on the parameters $\beta,\mu$ of~\eqref{eq: initial data 1}}.
 \end{Thm}

 \medskip

Extensions of this result to different interaction potentials have been recently achieved in~\cite{GSRT,PSS}.

\smallskip

As asserted by Boltzmann himself, the absence of contradiction between reversible microscopic (Newton) equations and the non-reversible Boltzmann equation is due to the fact that only particular, ``typical'' solutions to the former equation are well approximated by $f$. The way to give a precise meaning to this typicality is to introduce a statistical description of the initial state, which is in fact the point of view of Theorem \ref{lanford-thm}
 \cite{La76,S2}.
 
The goal of the present paper is to analyze in detail the proof of Lanford's theorem in order to point out where irreversibility shows up.
We shall see that part of the information is lost in the convergence process 
as some pathological sets of configurations with vanishing measure are neglected.
These sets turn out to be not time-reversal invariant and the possibility to retrace one's steps fades in the limit.

Furthermore note that, in Theorem \ref{lanford-thm}, the weak notion of convergence at time $t$ prevents 
us from iterating the result as written. Describing more precisely the geometry of the microscopic sets, we shall 
{use} a notion of {\bf one-sided convergence} holding at positive times as well as at time zero.
Thus we will obtain a refined statement of the theorem (Theorem~\ref{lanford-thm2}) compatible both with the {\bf irreversibility} and the 
{\bf time-concatenation} (semigroup) properties of the limiting equation (Sect.\;\ref{sec:ITC}).
A similar notion of one-sided convergence has been introduced by Denlinger in \cite{denlinger}, see also
\cite{Ki75} for a first, non quantitative version.

In order to characterize precisely the (small) sets where the convergence of the initial data is essential, we shall finally construct 
explicit examples of measures which are badly behaved exclusively in those regions, leading to a violation of Theorem 
\ref{lanford-thm} (Sect.\;\ref{slightmodif}).

\subsection{Microscopic dynamics}

\noindent In the following we denote, for~$1 \leq i \leq N$, $z_i:=(x_i,v_i)$ and~$Z_N:= (z_1,\dots,z_N)$.   With a slight abuse we say that~$Z_N $ belongs to~$ \T^{dN} \times \R^{dN}$ if~$ X_N:=(x_1,\dots,x_N)$ belongs to~$ \T^{dN}$ and~$V_N:=(v_1,\dots,v_N)$ to~$ \R^{dN}$.
The phase space is denoted by
$$
{\mathcal D}_\eps^N:= \big\{Z_N \in \T^{dN} \times \R^{dN} \, / \,  \forall i \neq j \, ,\quad |x_i - x_j| > \eps \big\} \,,
$$
{where $|\cdot|$ stands for the distance on the torus.}
We now distinguish   pre-collisional configurations  from   post-collisional  ones by defining 
 for indexes~$1 \leq i \neq j \leq N$
$$
\begin{aligned}
\d {\mathcal D}_{\eps}^{N\pm }(i,j) := \Big \{Z_N \in \T^{dN} \times \R^{dN} \, / \,  &|x_i-x_j| = \eps \, , \quad \pm (v_i-v_j) \cdot (x_i- x_j) >0 \\
& \mbox{and}  \, \forall (k,\ell) \in [1,N]^2 \setminus \{(i,j)\} \hbox{ with }  k\neq \ell \, ,  |x_k-x_\ell| > \eps\Big\} \, .
\end{aligned}
$$
Given a post-collisional configuration $Z_N$ on~$\d {\mathcal D}_{\eps}^{N+}(i,j)$,  we define~$Z_N'\in \d {\mathcal D}_{\eps}^{N-}(i,j)$ as the (pre-collisional) configuration having the same positions $(x_k)_{1\leq k\leq N}$, the same velocities~$(v_k)_{k\neq i,j}$ for non interacting particles, and the following pre-collisional velocities for particles~$i$ and $j$
$$
\begin{aligned}
 v_i'& := v_i - \frac1{\eps^2} (v_i-v_j)\cdot (x_i-x_j) (x_i-x_j)   \\
 v_j'& := v_j + \frac1{\eps^2} (v_i-v_j)\cdot (x_i-x_j) (x_i-x_j) \, .
\end{aligned}
$$

	\medskip
	
	\noindent 
Defining the Hamiltonian  
\begin{equation}
\label{eq: Hamilton}
H_N (V_N):= \frac12  \sum_{i=1}^N |v_i|^2  \, ,
\end{equation}
we consider the Liouville equation in the $2Nd$-dimensional phase space~${\mathcal D}_\eps^N$ 
	\begin{equation}
	\label{Liouville}
	\d_t f_N +\{ H_N, f_N\} =0\, ,
	\end{equation}
	with specular reflection on the boundary, meaning that if~$Z_N$ belongs to~$\d {\mathcal D}_{\eps}^{N+}(i,j)$  then
\begin{equation}\label{tracecondition}
	 f_N (t,Z_N ) =  f_N (t,Z_N') \, .
\end{equation}
We have denoted~$\{ \cdot, \cdot\}$ the Poisson bracket defined by
$$
\{ f, g\}:= \nabla_{V_N} f \cdot \nabla_{X_N} g -  \nabla_{X_N} f \cdot \nabla_{V_N} g \, .
$$
The Liouville equation~(\ref{Liouville})
writes therefore
$$
\d_t f_N + V_N \cdot   \nabla_{ X_N} f_N = 0\, , 
$$
with initial data given by \eqref{eq: initial data 2} and the condition \eqref{eq: initial data 1}.

\begin{Rmk}
\label{boundary-rmk}
Note that although the boundary condition~{\rm(\ref{tracecondition})} seems to introduce a symmetry between pre-collisional and post-collisional configurations, what has to be prescribed for the system to be well-posed is the density on post-collisional configurations for positive times, and for pre-collisional configurations for negative times, which are the incoming configurations for the transport equation~{\rm(\ref{Liouville})}.
\end{Rmk} 

We recall, as shown in~\cite{alexanderthesis} for instance, that
the set of initial configurations leading to   ill-defined characteristics (due to grazing collisions, clustering of collision times, or collisions involving more than two particles) is of measure zero in~${\mathcal D}_\eps^N$.

\subsection{Propagation of chaos}
\label{sec: chaos}

We define the marginals  on~$\mathcal D_\eps^n$ (extending by zero outside) by
\begin{equation}
\label{eq: marginals}
f_N^{(n)} (t,Z_n):=\int f_N (t, Z_N) \, dz_{n+1} \dots dz_N .
\end{equation}
Then one can show formally as in~\cite{grad,lanford} and~\cite{CIP,GSRT} that the first marginal, which describes 
the typical evolution of the gas, evolves according to  
\begin{align}
\label{eq: marginal 2}
& (\d_t +  v \cdot \nabla_x ) f_N^{(1)} (t,x,v) 
=  (N-1) \eps^{d-1} \\
& \times  \int_{{\mathbb S}^{d-1} \times \R^d} \Big( f_N^{(2)}( t, x , v',  x+\eps \nu, v_1')
 -   f_N^{(2)}( t, x, v,  x  - \eps \nu, v_1) \Big)
 \big((v_1-v) \cdot \nu\big)_+ d\nu dv_1 \, ,  
\nonumber
\end{align}
with $v',v_1'$ as in \eqref{boltzmann}.
This equation can be interpreted by saying that a particle at $z=(x,v)$ moves in a straight line until it collides with one of the remaining $N-1$ particles with velocity~$v_1$. The velocities $v',v_1'$ after the collision are then updated and the source term is determined by the joint distribution $f_N^{(2)}$.

\medskip

The notion of propagation of chaos (Sto{\ss}zahlansatz) lies at the heart of the derivation of Boltzmann's equation
 \eqref{boltzmann}. 
Heuristically, one would like to write that 
when two particles at configurations~$z= (x,v)$ and $z_1 = (x + \eps \nu, v_1)$ collide  then the marginal distribution factorizes
\begin{equation}
\label{eq: Stosszahlansatz 0}
\lim_{N \to \infty} \Big| f_N^{(2)} (t,z, z_1) -  f_N^{(1)}(t,z) f_N^{(1)}(t,z_1) \Big| = 0 \, .
\end{equation}
This statement of the Sto{\ss}zahlansatz is far from a mathematical assertion as $f_N^{(2)}$ is only defined almost surely in $\T^{2d} \times \R^{2 d}$ and not on sets of codimension 1. 
A more standard notion of propagation of chaos is given by the following definition.
\begin{Def}[Chaos property]
\label{def: chaos}
The sequence of measures $f_N$ is said asymptotically chaotic at time $t$ if there exists a measurable $f(t)$
on $\T^{d} \times \R^{d}$ such that, almost surely in 
$(z, z_1)$ in~$ (\T^{ d} \times \R^{d})^2$,
\begin{equation}
\label{eq: Stosszahlansatz}
 \begin{aligned}
&  \lim_{N \to \infty} f_N^{(1)}(t,z) = f(t,z) \;,\\
& \lim_{N \to \infty} \Big| f_N^{(2)} (t,z, z_1) -  f(t,z) f(t,z_1) \Big| = 0\;.
 \end{aligned}
\end{equation}
\end{Def}
In \eqref{eq: Stosszahlansatz} the coordinates $z, z_1$ are fixed independently of $N$ and $\eps$ (contrary to \eqref{eq: Stosszahlansatz 0}). As a consequence, this notion turns out to be too weak to derive Boltzmann equation from the microscopic evolution.

We shall see in Section \ref{sec: Lanford's proof} that the proof of Theorem \ref{lanford-thm} is not based on proving propagation of chaos but on a more global convergence of all the marginals. One of the goals of this paper is to quantify the refined notion of convergence (see Theorem \ref{lanford-thm2}) which is strictly needed in Lanford's argument. The propagation of chaos~\eqref{eq: Stosszahlansatz} can be derived as a byproduct.

 \section{Lanford's proof}
\label{sec: Lanford's proof}

  In order to understand how the assumptions on the initial data come into play, we have to look more precisely at the proof of  Theorem \ref{lanford-thm}. 
  Theorem \ref{lanford-thm} is actually the corollary of a more precise  result.
  Lanford's result indeed provides the convergence of all marginals
$f_N^{(n)}$  defined in \eqref{eq: marginals} to the solutions $f^{(n)}$ of an infinite system of coupled equations
\begin{equation}
\label{hierarchyboltzmann}
 \begin{aligned}
&  \d_t f^{(n)} +\sum_{i=1}^n v_i\cdot \nabla_{x_i} f^{(n)} =       C^0_{n,n+1} f^{(n+1)},\\
&\!\!\Big(   C^0_{n,n+1} f^{(n+1)}\Big) (x_1,v_1,\dots, x_n, v_n) \\
&     
:=\sum_{i=1}^n \iint_{{\mathbb S}^{d-1} \times \R^d} \Big(  f^{(n+1)} (x_1,v_1,\dots, x_i, v'_i,\dots x_n,v_n, x_i, v_{n+1}')\\
   &  \qquad\qquad   -  f^{(n+1)} (x_1,v_1,\dots, x_i, v_i,\dots x_n,v_n, x_i, v_{n+1})\Big)    \,\big(( v_{n+1}-v_i)  \cdot \nu\big)_+  \, dv_{n+1} d\nu  \, ,\\
 \end{aligned}
 \end{equation}
which is the so-called Boltzmann hierarchy. 
Chaotic families   of the form $ f^{(n)}= f^{\otimes n}$ with $f$ solution to the Boltzmann equation are specific solutions to this hierarchy, where
$$
f^{\otimes n} (Z_n):= \prod_{i=1}^n f(z_i)\;.
$$
The connection between Boltzmann hierarchy and Boltzmann equation is discussed in \cite{S1}.

 The starting point of the proof is    to write an explicit representation of the~$n$ particle distribution 
 $f_N^{(n)}$ as a superposition of different~$(n+s)$-particle pseudo-dynamics, with weights depending on the initial data. More precisely, by averaging and iterating Duhamel's formula for the $N$-particle distribution $f_N$, 
 we  end up with a series expansion for $f_N^{(n)}$ in which the term of order $s$ corresponds to   pseudo-dynamics involving $s$ collisions and 
 is therefore expressed as an operator acting on the initial $(n+s)$-particle distribution $f_{N,0}^{(n+s)}$ (see Section~\ref{seriesexpansion}).
 
 The strategy of proof then relies  on two main steps.
 \begin{itemize}
 \item First we obtain a uniform bound on the series expansion, which explains the short time restriction in Theorem \ref{lanford-thm}. 
In the following, we restrict our attention to times smaller than the radius of analyticity of the series.
 \item The convergence to the solution of the Boltzmann hierarchy  then follows from the convergence of the trajectories representing the different pseudo-dynamics
  (note that these trajectories are related to the representation formula and that they do not coincide in general with the physical trajectories of the particles,
  e.g.\,\cite{PS15} for further discussions). The convergence of pseudo-trajectories fails to hold when there are recollisions (see page~\pageref{defrecollision} for a precise definition of recollisions). A geometric argument shows however that, for any fixed~$n$, the set of initial configurations with
  $n$ particles leading to such recollisions is of vanishing measure in the~$N \to \infty$ limit.
 \end{itemize}
Note that all the information on these bad sets is forgotten in the limit: this is related to   irreversibility, that is to the impossibility of going back to the initial state. Furthermore the convergence of the first marginal to the solution of the Boltzmann equation in the case of factorized initial data such as~(\ref{eq: initial data 2})
 is due to a uniqueness property for the Boltzmann hierarchy; this follows from the uniform bound on the hierarchy obtained in the first step of the above strategy.

\subsection{The series expansions}
\label{seriesexpansion}

 A formal computation based on Green's formula (see~\cite{Ce72,lanford,GSRT}  for instance)  leads to the following  BBGKY  hierarchy for~$n<N$
\begin{equation}
\label{eq: BBGKY}
(\d_t +\sum_{i=1}^n v_i\cdot \nabla_{x_i} ) f_N^{(n)} (t,Z_n) =   \big( C_{n,n+1} f_N^{(n+1)}\big) (t,Z_n)\, ,
\end{equation}
	on~$\cD_\eps^n$ 	with the boundary condition as in~(\ref{tracecondition})
	$$f_N^{(n)} (t,Z_n) = f_N^{(n)} (t,Z_n') \quad \text{ on } \quad \d {\mathcal D}_\eps^{n +}(i,j)  \, .$$
  	The collision term is defined by
\begin{align}
\label{BBGKYcollision}
\big( C_{n,n+1} & f_N^{(n+1)}\big) (Z_n)  :=   (N-n) \eps^{d-1} \nonumber \\
&\times \Big( \sum_{i=1}^n \int_{{\mathbb S}^{d-1} \times \R^d}  f_N^{(n+1)}(\dots, x_i, v_i',\dots , x_i+\eps \nu, v'_{n+1}) \big((v_{n+1}-v_i) \cdot \nu\big)_+ d\nu dv_{n+1}  \\
& \quad    - \sum_{i=1}^n \int_{{\mathbb S}^{d-1} \times \R^d}    f_N^{(n+1)}(\dots, x_i, v_i,\dots , x_i+\eps \nu, v_{n+1}) \big((v_{n+1}-v_i) \cdot \nu\big)_- d\nu dv_{n+1} \Big), \nonumber \\
\mbox{with} & \quad v_i ' := v_i  -  (v_i -v_{n+1})\cdot \nu \,  \nu \,   ,\quad v_{n+1}' := v_{n+1} +  (v_i -v_{n+1})\cdot \nu \,  \nu \, . \nonumber
\end{align}
The closure  for $n=N$ is given by the Liouville equation (\ref{Liouville}). 
				  Note that  the collision integral is split into two terms according to the sign of $(v_i-v_{n+1}) \cdot \nu $ and we used the trace condition on $\d {\mathcal D}_\eps^{N+}(i,n+1)$
		to  express all quantities in terms of pre-collisional configurations.

To obtain the Boltzmann hierarchy, we  compute the formal  limit of the transport and collision operators when~$\eps$ goes to~$ 0$. 
Recall that for fixed $n$, then $(N-n) \eps ^{d-1} \to 1$ in the Boltzmann-Grad limit.
Thus the limit hierarchy is  given by
\begin{equation}
\label{hier: Botzmann}
	(\d_t +\sum_{i=1}^n v_i\cdot \nabla_{x_i} ) f^{(n)} (t,Z_n) =   \big(   C^0_{n,n+1} f^{(n+1)}\big) (t,Z_n)
	\end{equation}
in~$(\T^d \times \R^d)^n$, where~$  C^0_{n,n+1}$ are   the limit collision operators defined by (\ref{hierarchyboltzmann}). We denote by~$(f_0^{n})_{n \in \N}$ a family of initial data for this hierarchy (which will be specified later).

\medskip
		
Iterating Duhamel's formula for the BBGKY hierarchy (\ref{eq: BBGKY}), we get
\begin{equation}
\label{duhamel1}
 \begin{aligned}
 f^{(n)} _N(t) =\sum_{s=0}^{N-n}    Q_{n,n+s}(t) f^{(n+s)}_{N,0} \, ,
\end{aligned}
\end{equation}
where we have defined
$$ \begin{aligned}
Q_{n,n+s}(t) f^{(n+s)}_{N,0}  := \int_0^t \int_0^{t_{n+1}}\dots  \int_0^{t_{n+s-1}}  {\bf S}_n(t-t_{n+1}) C_{n,n+1}  {\bf S}_{n+1}(t_{n+1}-t_{n+2}) C_{n+1,n+2}   \\
\dots  {\bf S}_{n+s}(t_{n+s})     f^{(n+s)}_{N,0} \: dt_{n+s} \dots d t_{n+1} 
\end{aligned}$$
denoting by~${\bf S}_n$  the group associated with free transport in $\cD_\eps^n$ with specular reflection on the boundary. 

\begin{Rmk}\label{collision-rmk}
Note that, for fixed $N$,  the operator $C_{n,n+1}  $ is  a trace on a manifold of codimension 1 and thus it is a priori not defined  on  $L^\infty$ functions. 
What makes sense is the combination $\displaystyle \int dt_{n+1} C_{n,n+1}  {\bf S}_{n+1}(t_{n+1}-t_{n+2})$ (see~\cite{bardos, Simonella, GSRT} and Figure~{\rm\ref{volume-fig}}).

\begin{figure}[h]
\begin{center}
\scalebox{0.4}{\includegraphics{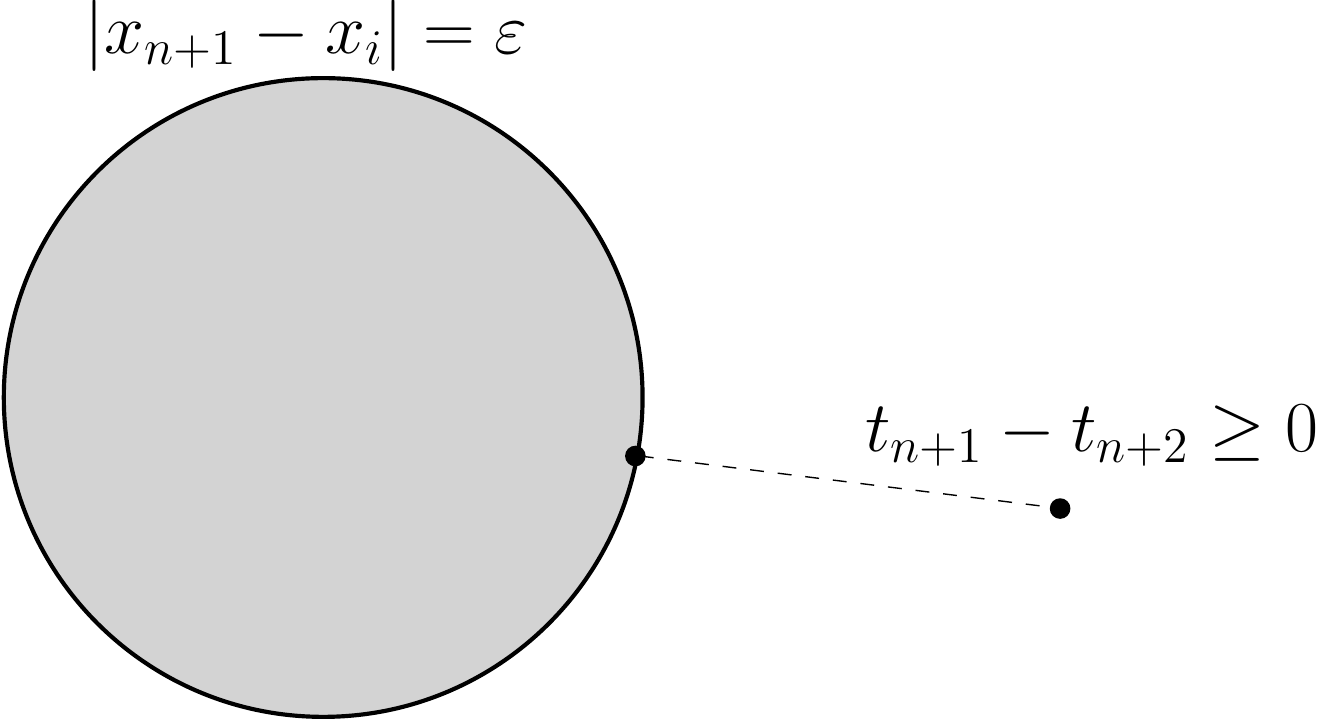}}
\caption{\label{volume-fig} 
{
The grey domain is an excluded region and its boundary is the surface 
$|x_{n+1} - x_i| = \eps$ corresponding to a  collision  between particles $i$ and $n+1$. 
The admissible configurations are outside this domain and can be parametrised
by a point of the surface and a non negative time $t_{n+1}-t_{n+2}$, provided that
the velocities are pre-collisional.}
}
\end{center}
\end{figure}

For $t\geq 0$, one has $t_{n+1}-t_{n+2}\geq 0$, it is therefore necessary to express the collision operator in terms of pre-collisional configurations. 
In a symmetric way, for $t\leq 0$, one has~$t_{n+1}-t_{n+2}\leq 0$, and we have to express the collision operator in terms of post-collisional configurations (see Remark~{\rm\ref{boundary-rmk}}).
\end{Rmk}

Similarly, for the Boltzmann hierarchy (\ref{hier: Botzmann})
 \begin{equation}
 \label{duhamel2}
 \begin{aligned}
 f^{(n)}(t) =\sum_{s=0}^{\infty}    Q^0_{n,n+s}(t) f^{(n+s)}_{0} \, ,
\end{aligned}
\end{equation}
where we have defined
$$ \begin{aligned}
Q^0_{n,n+s} (t) f^{(n+s)}_{0} = \int_0^t \int_0^{t_{n+1}}\dots  \int_0^{t_{n+s-1}}  {\bf S}^0_n(t-t_{n+1})   C^0_{n,n+1}  {\bf S}^0_{n+1}(t_{n+1}-t_{n+2})   C^0_{n+1,n+2}   \\
\dots  {\bf S}^0_{n+s}(t_{n+s})     f^{(n+s)}_{0} \: dt_{n+s} \dots d t_{n+1} ,
\end{aligned}$$
denoting by~${\bf S}^0_n$  the group associated with free transport in $(\T^d \times \R^d)^n$.

Let us denote $|{C}_{s,s+1} |$, $|Q_{n,n+s}|$ the operators obtained by summing the absolute values of all the elementary terms.
The energy $H_s =\frac 12 \sum_{i=1}^{s} |v_i|^2$  is conserved by the transport so that
$$
{\bf S}_s \left(  \exp\left(- \beta H_s \right) \indc_{\cD_\eps^s} \right) =  \exp\left(-\beta H_s \right) \indc_{\cD_\eps^s} \, ,
$$
and from the loss estimates on the collision operators (see \cite{GSRT} for instance)
$$
\begin{aligned}
|{C}_{s,s+1} | 
\left( \exp\left(-\beta H_{s+1}\right) \indc_{\cD_\eps^{s+1} } \right) 
  \leq C {\beta^{-d/2} } \Big( s\beta^{-\frac12} + \sum_{1 \leq i \leq s} |v_i|\Big) 
\exp\left(-\beta H_s \right) \indc_{\cD_\eps^s} \, ,
\end{aligned}
$$
we  get the Cauchy-Kowalevsky type iterated estimate  for~$\tilde \beta<\beta$ 
\begin{equation}
\label{continuity2}
 |Q_{n,n+s}| (t)\left(  \exp\left(-\beta  H_{n+s}\right) \indc_{\cD_\eps^{n+s} } \right)  \leq  C^{n+s} C_{\beta, \tilde \beta } ^{s}   \exp\left(-\tilde \beta H_{n}  \right)\,,
\end{equation}
with $C_{\beta, \tilde \beta} = \beta^{-(d+1)/2}   t/( \beta- \tilde \beta)$.

{\color{black}
Using the initial data \eqref{eq: initial data 2} and the condition \eqref{eq: initial data 1}, we deduce 
following~\cite{ukai} an upper bound on the marginals from \eqref{duhamel1} 
\begin{equation}
\label{eq: convergence}
\forall t \leq t^*, \qquad 
f_N^{(n)} (t)  \leq    \exp  (( \lambda t- \mu ) n  ) \exp \Big(( \lambda t- \beta )  H_n \Big) \, ,
\end{equation}
where~$\lambda$, chosen large enough,
 depends on~$\beta, \mu$, and~$t^* $ is such that~$\lambda t^* = \beta/2$.
The convergence time in Lanford's Theorem \ref{lanford-thm} is given by $t^*$.
}

Similar estimates hold for the limit operators~$Q^0_{n,n+s}$ and~$ {\bf S}_s^0$, as well as for the solution of the Boltzmann hierarchy.

\subsection{Geometrical representation as a superposition of pseudo-dynamics}

The usual way to study the $s$-th term of the representation formula is to  introduce some pseudo-dynamics describing the action of the operator $Q_{n, n+s}$.
We first extract combinatorial information on the collision process: we describe the adjunction of new particles (in the backward dynamics)   by  ordered trees.

 \begin{Def}[Collision trees]
 \label{trees-def}
 Let $n\geq 1$, $s\geq 1$ be fixed. An (ordered) collision tree  $a \in \cA_{n, n+s} $ is defined by a family $(a(i)) _{n+1\leq i \leq n+s}$ with $a(i) \in \{1,\dots, i-1\}$.
 \end{Def}
 Note that~~$|\cA_{n,n+s}| \leq n(n+1) \dots (n+s-1)$.

\medskip

  Once we have fixed a collision tree $a \in \cA_{n, n+s}$, we can reconstruct pseudo-dynamics starting from any point in the $n$-particle phase space $Z_n= (x_i, v_i) _{1\leq i\leq n}$ at time $t$.

\begin{Def}[Pseudo-trajectory]
\label{pseudotrajectory}
Given~$Z_n  \in \cD_\eps^n$, consider a collection of times, angles and velocities~$(T_{n+1,n+s}, \Omega_{n+1,n+s}, V_{n+1,n+s}) = (t_i, \nu_i, v_i)_{n+1\leq i\leq n+s}$ with~$0\leq t_{n+s}\leq\dots\leq t_{n+1}\leq t$. We then define recursively the pseudo-trajectories in terms of the  backward BBGKY dynamics as follows
\begin{itemize}
\item in between the  collision times~$t_i$ and~$t_{i+1}$   the particles follow the~$i$-particle backward flow with specular reflection;
\item at time~$t_i^+$,  particle   $i$ is adjoined to particle $a(i)$ at position~$x_{a(i)} + \eps \nu_i$ provided it remains at a distance~$\eps$  from all the others,
 and 
with velocity~$v_i$. 
If $(v_i - v_{a(i)} (t_i^+)) \cdot \nu_i >0$, velocities at time $t_i^-$ are given by the scattering laws
\begin{equation}
\label{V-def}
\begin{aligned}
v_{a(i)}(t^-_i) &= v_{a(i)}(t_i^+) - (v_{a(i)}(t_i^+)-v_i) \cdot \nu_i \, \nu_i  \, ,\\
v_i(t^-_i) &= v_i+ (v_{a(i)}(t_i^+)-v_i) \cdot \nu_i \,  \nu_i \, .
\end{aligned}
\end{equation}
\end{itemize}
 We denote  by~$z_i(a, Z_n, T_{n+1,n+s}, \Omega_{n+1,n+s}, V_{n+1,n+s},\tau)$ the position and velocity of the particle labeled~$i$, at time~$\tau$ (provided~$\tau< t_i$).  We also define
$
G_{n+1,n+s}(a) $ as the set of parameters $(T_{n+1,n+s}, \Omega_{n+1,n+s}, V_{n+1,n+s}) $ such that the pseudo-trajectory  $ Z_{n+s}(a, Z_n, T_{n+1,n+s}, 
\Omega_{n+1,n+s}, V_{n+1,n+s},\tau )$ exists up to time 0, meaning that by adjunction of a new particle, there is no overlap. The  configuration obtained at the end of the tree, i.e. at time 0, is~$Z_{n+s}(a, Z_n, T_{n+1,n+s}, \Omega_{n+1,n+s}, V_{n+1,n+s},0)$.

 Similarly, we define  the pseudo-trajectories associated with the Boltzmann hierarchy. 
These pseudo-trajectories evolve according to the backward Boltzmann dynamics as follows
\begin{itemize}
\item in between the  collision times~$t_i$ and~$t_{i+1}$   the particles follow the~$i$-particle backward free flow;
\item at time~$t_i^+$,  particle   $i$ is adjoined to particle $a(i)$  at exactly the same position $x_{a(i)}$. Velocities are given by the laws {\rm(\ref{V-def})}.
\end{itemize}
We denote this flow by $Z^0_{n+s}(a, Z_n, T_{n+1,n+s}, \Omega_{n+1,n+s}, V_{n+1,n+s},\tau)$.
\end{Def}

With these notations, the representation formulas (\ref{duhamel1}) and (\ref{duhamel2}) for the marginals of order~$n$ can be rewritten respectively
$$
 \begin{aligned}
 f^{(n)} _N(t, Z_n) =\sum_{s=0}^{N-n}   &C^{(N,n,s)} \sum_{a \in \cA_{n, n+s} }   \int_{G_{n+1,n+s}(a)} dT_{n+1,n+s}  d\Omega_{n+1,n+s}  dV_{n+1,n+s} \\
 &   
\Big( \prod_{i=n+1}^{n+s}  \big( (v_i -v_{a(i)} (t_i)) \cdot \nu_i \Big)  
f_{N,0}^{(n+s)}\big (Z_{n+s}(a, Z_n, T_{n+1,n+s}, \Omega_{n+1,n+s}, V_{n+1,n+s},0)\big)  \, ,
 \end{aligned}
$$
with $C^{(N,n,s)}:= (N-n) \dots\big(N-(n+s-1)\big) \eps^{(d-1)s}  = 1 + O ((n+s)^2/N) $ and
$$
 \begin{aligned}
 f^{(n)} (t, Z_n) =\sum_{s=0}^{\infty}   &  \sum_{a \in \cA_{n, n+s} }   \int_{\cT_{n+1,n+s}} dT_{n+1,n+s} \int_{{\mathbb S}^{s}} d\Omega_{n+1,n+s}  \int_{\R^{2s}} dV_{n+1,n+s} \\
 &\times 
\Big( \prod_{i=n+1}^{n+s}  \big( (v_i -v^0_{a(i)} (t_i)) \cdot \nu_i \Big)  
f_{0}^{(n+s)}\big (   Z^0_{n+s}(a, Z_n, T_{n+1,n+s}, \Omega_{n+1,n+s}, V_{n+1,n+s},0)\big)  \, ,
 \end{aligned}
$$
denoting $\cT_{n+1,n}=\emptyset$ and 
$$
\cT_{n+1,n+s} :=\big\{ (t_i)_{n+1\leq i\leq n+s} \in [0,t]^{s} \,/\, 0\leq t_{n+s}\leq\dots\leq t_{s+1}\leq t\big\}\,.
$$

\bigskip
 The question is then to describe the asymptotic behavior of the BBGKY pseudo-trajectories. We actually split them into two classes~:
\begin{itemize}\label{defrecollision}
\item   pseudo-trajectories having no recollision, i.e. such that particles interact only at the times of adjunction of new particles, and are transported freely between two such times;
\item   pseudo-trajectories involving recollisions.
\end{itemize}
Note that no recollision occurs in the Boltzmann hierarchy as the particles have zero diameter.

\subsection{Bad configurations}
 The transport semigroups ${\bf S}_n$ (with recollisions) and ${\bf S}^0_n$ (free transport)  play a key role in the discrepancies between the BBGKY series \eqref{duhamel1} and the Boltzmann series \eqref{duhamel2}. 
In a given time interval, both transports coincide if no recollision occurs which will be the typical case for fixed $n$ and $\eps$ small. However,  specific configurations lead to recollisions and we   define below the corresponding geometric sets.

Denote by~$B_R$ the ball of~$\R^d$ centered at zero and of radius~$R$, and fix a time~$T$ much bigger than the radius of analyticity $t^*$ given in~\eqref{eq: convergence} as well as a parameter~$\eps_0 \gg \eps$. The sets of bad configurations of $n$ particles are defined as
\begin{equation}
\label{eq: bad set}
\begin{aligned}
 \cB_{\eps_0} ^{n -}:= \Big\{ Z_n \in (\T^d \times B_R)^n , \qquad  \exists u \in [0,T] \, ,  \, \exists i,j , \qquad |x_i - x_j - u(v_i - v_j) | \leq \eps_0  \Big\} \, ,\\
  \cB_{\eps_0}^{n +}:= \Big\{ Z_n \in (\T^d \times B_R)^n , \qquad  \exists u \in [0,T] \, ,  \, \exists i,j , \qquad |x_i - x_j + u(v_i - v_j) | \leq \eps_0  \Big\}\,,
  \end{aligned}
\end{equation}
{where $|\cdot|$ stands for the distance on the torus.}
This  means that, starting from~$\cB_{\eps}^{n -}$ at time~$t$, the backward free flow on $\cD_\eps^n$ will involve at least one recollision between~$t$ and~$t-T$, and starting from $\cB_\eps^{n +}$, the forward free flow on $\cD_\eps^n$ will involve at least one recollision between~$t$ and $t+T$.
In particular outside these sets, we have
$$ \begin{aligned}
\left( {\bf S}_n (t)- {\bf S}^0_n (t)\right)\Big(  f_{N,0} ^{(n)}(1-\indc_{\cB_\eps ^{n +}})\Big)  &= 0
\quad \hbox{ for } t \in [0,T]\, ,\\
\left( {\bf S}_n (t)- {\bf S}^0_n (t)\right)\Big(   f_{N,0} ^{(n)}(1-\indc_{\cB_\eps ^{n -}})\Big)  &= 0
\quad \hbox{ for } t \in [-T,0]\, .
\end{aligned}
$$
The first term in each series \eqref{duhamel1} and \eqref{duhamel2} involves the transport, both first terms coincide when $\pm t >0$ for configurations which are outside the bad set $\cB_\eps ^{n \pm}$.
{We stress the fact that similar sets have been already introduced by Denlinger in \cite{denlinger}
and previously identified in \cite{BLLS} as key sets
(see Appendix~A of \cite{BLLS} for a discussion on the irreversibility).

\medskip

The following result is an easy calculation.
\begin{Prop}
The bad sets are ordered
$$ \cB_{\eps}^{n\pm} \subset \cB_{\eps'} ^{n\pm} \quad \forall \eps' \geq \eps\,.$$
Their measure is controlled by
$$|\cB_\eps^{n\pm} | \leq (CR^d)^{n} n^2 RT \eps^{d-1}\,,$$
and the intersection is much smaller
$$|\cB_\eps^{n+}\cap \cB_\eps^{n-}  | \leq (CR^d)^{n}  \big( n^2  \eps^{d}+n^4 R^2 T^2 \eps ^{2(d-1)} \big) \,.$$
\end{Prop}

\bigskip
We now suppose that  $t\geq 0$ since the situation when $t\leq 0$ can be deduced by a simple symmetry in~$t$ and~$v$.
The next terms  in the series expansion \eqref{duhamel1}, \eqref{duhamel2}
involve some averaging with respect to the parameters $(t_i, v_i, \nu_i)_{n+1\leq i\leq n+s}$ describing the adjunction of   new particles. 
What can be proved is that, provided that the~$n$-particle backward flow~${\bf \Psi}_n $  on $\cD_\eps^n$
 does not lead to a recollision, then the probability of having a recollision (involving at least one of  the added particles) is very small.

\subsection{Convergence outside   bad configurations}
Let us first prove that the solutions are close by eliminating bad trajectories.
By definition, the set of good configurations with $k$ particles will be such that the particles remain, by backward free flow, at a distance $\eps_0 \gg \eps|\log \eps|$ for a time $T\gg t^*$, i.e. that they belong to the set  
$$
\cG_k(\eps_0) := (\T^d \times B_R)^k \setminus  \cB_{\eps_0} ^{k-} \, .
$$
For particles in $\cG_k(\eps_0)$, the transport ${\bf \Psi}_k$ on $\cD_\eps^k$  coincides with the free flow $ {\bf \Psi}^0_k$ on $(\T^d \times B_R) ^k$. 
Thus, if at time $t$ the configurations $Z_k$, $Z_k^0$ are  such that 
\begin{equation}
\label{eq: example}
\forall i \leq k \, , \qquad |x_i - x_i^0| \leq \eps|\log \eps|\, ,
\qquad 
v_i = v_i^0
\end{equation}
and   $Z_k^0$ belongs to $\cG_k(\eps_0)$, then the configurations ${\bf \Psi}_k (u) Z_k$, ${\bf \Psi}^0_k (u) Z_k^0$ will remain at distance less than $O(\eps|\log \eps|)$ for $u \in [0,t]$.
{Recall that the distance $|\cdot|$ is on the torus.}

 One can show that   good configurations are stable by adjunction of a 
$(k+1)^{\text{th}}$-particle next to a particle labelled by $m_k \leq k$, provided some bad sets are removed. 
More precisely, let~$Z_k^0 = (X_k^0, V_k)$ be in~$\cG_k(\eps_0)$ and  
$Z_k = (X_k, V_k)$ with positions close to $X_k^0$ and same velocities (cf. \eqref{eq: example}). Then,
by  choosing the velocity $v_{k+1}$ and the deflection angle 
$\nu_{k+1}$ of the new particle $k+1$ outside a bad set $B_{m_k} (Z_k^0)$, both configurations
$Z_k$ and $Z_k^0$ will remain close to each other.
Of course, immediately after the adjunction, the particles $m_k$ and $k+1$ will not be at distance~$\eps_0$, but~$v_{k+1}, \nu_{k+1}$ can be chosen such that the particles drift rapidly far apart and
after a short time $\delta>0$ the configurations $Z_{k+1}$ and $Z^0_{k+1}$
are again in the good sets $\cG_{k+1} (\eps_0/2)$ and~$\cG_{k+1} (\eps_0)$.

\begin{Prop}[\cite{GSRT}]
\label{geometric-prop}
We   fix   parameters~$ \eps \ll \eps_0, \delta \ll 1$     such that
\begin{equation}
\label{sizeofparameters}
|\log \eps | \eps \ll \eps_0 \ll \min(\delta R,1)    \, .
\end{equation}
Given~$Z_k^0 = (X_k^0, V_k) \in \cG_k(\eps_0)$ and  $m_k \leq k$, there is a 
subset~$B_{m_k}  (Z_k^0)$ 
of~${\mathbb S}^{d-1} \times B_R$  of small measure
\begin{equation}
\label{pathological-size}
\big |B_{m_k}( Z_k^0) \big |  \leq  CkR^d 
\gamma (\eps,\eps_0)
\quad \text{with} \quad 
\gamma (\eps,\eps_0):= |\log \eps|  \left( \left( { \eps \over \eps_0}\right)^{d-1} + (R T  )^d  \eps_0  ^{d-1}+\left({ \eps_0 \over R \delta}\right)^{d-1}  \right),
\end{equation}
such that good configurations close to $Z_k^0 $  are stable by adjunction of a collisional particle close to the particle~$x^0_{m_k}$ 
in the following sense. 
\smallskip
\noindent 	 
Let~$Z_k = (X_k, V_k)$ be a configuration of $k$ particles satisfying \eqref{eq: example}, i.e.
$| X_k- X_k^0  | \leq |\log \eps | \eps$. 
Given~$(\nu_{k+1},v_{k+1}) \in ({\mathbb S}^{d-1} \times  B_R) \setminus B_{m_k}( Z_k^0)$,
a new particle with velocity $v_{k+1}$ is added at $x_{m_k} + \eps \nu_{k+1}$ to $Z_k$ and at $x_{m_k}^0$ to $Z_k^0$.
Two possibilities may arise 

$ \bullet $ For a pre-collisional configuration~$\nu_{k+1} \cdot (v_{k+1} - v_{m_k})< 0$ then 
\begin{equation}
\label{taugeq0precoll}
\forall u \in ]0,t]  \, , \quad \left\{ \begin{aligned}
& \forall i \neq j \in [1,k] \, , \quad  | (x_i -u \, v_i) -  (x_j - u \, v_j) |  > \eps \, ,\\
& \forall j \in [1,k] \, , \quad  | (x_{m_k} + \eps \nu_{k+1} - u \,v_{k+1}) -  (x_j - u\,  v_j) |  > \eps \, .
\end{aligned}
\right.
\end{equation}
Moreover after the time~$\delta$, the~$k+1$ particles are in a good configuration
\begin{equation}
\label{taugeqdelta2precoll}
\forall u \in [\delta ,t]  \, , \quad \left\{ \begin{aligned}
& (X_k - u V_k   , V_k , x_{m_k} + \eps \nu_{k+1} - u \, v_{k+1} , v_{k+1}) \in \cG_{k+1} (\eps_0/2),\\
& (X^0_k - u V_k   , V_k ,  x^0_{m_k} - u \, v_{k+1} , v_{k+1}) \in \cG_{k+1} (\eps_0)\, .
\end{aligned}
\right.
\end{equation}

$ \bullet$ For a post-collisional configuration~$\nu_{k+1} \cdot (v_{k+1} - v_{m_k}) > 0$ then 
the velocities are updated 
\begin{equation}\label{taugeq0precoll}
\forall u \in ]0,t]  \, , \quad \left\{ 
\begin{aligned}
&\forall i \neq j \in [1,k] \setminus \{m_k\} \, , \quad  | (x_i - u \, v_i) - ( x_j - u \, v_j ) | > \eps \, ,\\
& \forall j \in [1,k] \setminus \{m_k\} \, , \quad | (x_{m_k} + \eps \nu_{k+1} - u \, v_{k+1}' )- ( x_j - u \, v_j )|  > \eps \, ,\\
&\forall j \in [1,k] \setminus \{m_k\} \, , \quad | (x_{m_k}  - u \, v_{m_k}' )- (x_j - u \,  v_j )|  > \eps \, ,\\
&
| (x_{m_k}  - u \, v_{m_k}') - ( x_{m_k} + \eps \nu_{k+1} - u \, v_{k+1}')| > \eps 
\, .
\end{aligned}
\right.
\end{equation}
Moreover after the time~$\delta$, the~$k+1$ particles are in a good configuration
\begin{align}\label{taugeqdelta2precoll}
&\forall u\in [\delta, t], \nonumber \\
&\left\{ \begin{aligned}
& \big( \{ x_{j} - u \, v_{j}, v_j  \}_{j \not = m_k}, x_{m_k} - u \, v_{m_k}' ,v_{m_k}',  x_{m_k} + \eps \nu_{k+1} - u \, v_{k+1}' , v_{k+1}' \big) \in \cG_{k+1} (\eps_0/2),\\
& \big( \{  x^0_{j} - u \, v_{j}, v_j  \}_{j \not = m_k},  x^0_{m_k}- u \, v_{m_k}' ,v_{m_k}',  
 x^0_{m_k}  - u \, v_{k+1}' , v_{k+1}'\big) 
\in \cG_{k+1} (\eps_0)\, .
\end{aligned}
\right.
\end{align}
\end{Prop} 

 We refer to \cite{GSRT} for a complete proof of Proposition \ref{geometric-prop} and simply recall that it can be obtained from  the following control on free trajectories (note that compared to~\cite{GSRT} there is an additional loss of a $|\log \eps|$ which is due to the action of the scattering operator and is actually missing in~\cite{GSRT}).  
\begin{Lem}\label{geometric-lem1}
  Given~$T>0$,   $ \eps\ll \delta\ll 1$ and $ \eps |\log \eps |  \ll \eps_0\ll  \min(\delta R, 1)$,
   consider two points~$x^0_1,x^0_2$ in~$\T^d$ such that $|x^0_1 - x^0_2|\geq \eps_0$,  and a velocity~$ v_1 \in B_R$.
 Then there exists a subset~$K(x^0_1- x^0_2, \eps_0, \eps) $ of~$\R^d$   
 with  measure bounded by
 $$ |K( x^0_1- x^0_2, \eps_0, \eps) | \leq  CR^d  |\log \eps | \left(  \left(\frac{ \eps}{\eps_0}\right)^{d-1} +  (R t  )^d \,   { \eps} ^{d-1}\right) $$
 and a subset $K_\delta( x^0_1- x^0_2, \eps_0,\eps) $ of~$\R^d$, the measure of which satisfies
 $$  
 |K_\delta(x^0_1- x^0_2, \eps_0, \eps) | \leq CR^d |\log \eps | 
  \left(  
  \left(\frac{\eps_0 }{R\delta}\right)^{d-1} 
  +  (R t  )^d  \eps_0  ^{d-1}
 \right) 
 $$
 such that for any~$ v_2\in B_R$ and~$x_1, x_2$ such that $|x_1 - x^0_1| \leq |\log \eps | \eps$, $|x_2 - x^0_2| \leq |\log \eps | \eps$, the following results hold :

\noindent
 $ \bullet $ $ $ If   $v_1 -v_2 \not \in K( x^0_1- x^0_2, \eps_0, \eps) $, then 
  $$
   \forall u \in [0,t]  \,, \quad |(x_1-  u \, v_1) -  (x _2 - u \, v_2)| > \eps  .
   $$
 $ \bullet $ $ $ If  $v_1 -v_2 \notin K_\delta( x^0_1- x^0_2, \eps_0,  \eps) $
   $$
   \forall u \in  [\delta, t]   \,, \quad | ( x_1- u \, v_1) - ( x _2- u \, v_2) | >\eps_0 \, .
 $$
 \end{Lem}

 \bigskip
 \noindent Proposition \ref{geometric-prop} is the elementary step for adding a new particle. 
This step can be iterated 
in order to build inductively good pseudo-trajectories~$Z$ and $Z^0$.
Note that after adding a new particle,   velocities remain identical at each time in both configurations, but their positions differ due the exclusion condition in the BBGKY hierarchy which induces a shift of~$\eps$ at each creation of a new particle.

To estimate $Q_{n,n+s} (t)f^{(n+s)} _{N,0}  -Q^0_{n,n+s} (t)f^{(n+s)} _{0}$, we then split the integration domain in several pieces
\begin{itemize}
\item pseudo-trajectories with large energy $ H_{n+s} (Z_{n+s} ) \geq R^2 \gg 1$;
\item pseudo-trajectories  with collisions separated by less than a time $\delta \ll 1$;
\item pseudo-trajectories (with moderate energy and  collisions well separated in time) having recollisions;
\item good pseudo-trajectories in the sense of Proposition \ref{geometric-prop}.
\end{itemize}
Bad pseudo-trajectories have a small contribution to the integrals thanks to~(\ref{pathological-size})
while good pseudo-trajectories  of the BBGKY and Boltzmann hierarchies can be coupled.

\medskip

\subsection{Convergence of initial data}
To estimate the contribution of good pseudo-trajectories, we have then to combine the continuity of $f_0^{(n+s)}$ together with an estimate on the difference~$f^{(n+s)} _{N,0} - f^{(n+s)} _{0}$  between initial data  on the set of initial configurations which may be reached by such  pseudo-dynamics: since  we only   consider pseudo-trajectories leading to good configurations, what  we need to compute is~$(f_{N,0}^{(s)} - f_0^{\otimes s} ) (1- \indc_{\cB_{\eps} ^{s+} \cap \cB_{\eps} ^{s-}} )$. 

With the specific choice of initial data  (\ref{eq: initial data 2}) in Theorem {\rm \ref{lanford-thm}}, one can prove (see \cite{GSRT} for instance)  that  the initial data of both hierarchies are close, in the sense that
for $s \geq 2$
\begin{equation}\label{closedata}
\Big| (f_{N,0}^{(s)} - f_0^{\otimes s} ) (1- \indc_{\cB_{\eps} ^{s+} \cap \cB_{\eps} ^{s-}} ) \Big| 
\leq  C^s \exp (-\beta H_s) \eps   \,.
\end{equation}
{This condition implies that  $f_{N,0}^{(s)}$ is almost chaotic on the singular sets $\cB_{\eps} ^{s+} \setminus \cB_{\eps} ^{s-} $ (which are relevant for the forward equation) and~$\cB_{\eps} ^{s-} \setminus \cB_{\eps} ^{s+}$
(which are relevant for the backward equation).}
Note that, compared to Definition {\rm\ref{def: chaos}}, this is much stronger as it provides a quantitative description of the factorization {on sets depending on $\eps$.}

\medskip

It remains to gather all error estimates and to use the continuity property (\ref{continuity2}) for the 
operators~$Q_{n,s+n}$. 
We define the weighted norm 
$$
\|f_n\|_{L^\infty_{\beta,n}}:=  \|f_n  \exp (\beta H_n  )  \|_{L^\infty} \, ,
$$
with $H_n$ the Hamiltonian \eqref{eq: Hamilton}.
Fixing the parameters~$ \eps_0, \delta,s,n$  such that
$$
|\log \eps | \eps \ll \eps_0 \ll \min(\delta R,1)    \, ,\quad n+s \leq |\log \eps |\,,
$$
and choosing~$R \leq C |\log \eps|$, 
 the error term  $\gamma(\eps, \eps_0)$ from Proposition~\ref{geometric-prop} converges to 0. }  
The term by term convergence is then obtained from the following estimate, thanks to the previous analysis and~(\ref{closedata}).  
 \begin{Prop}\label{error-est}
 There is~$\gamma>0$ such that
$$\begin{aligned}
&\left\| \left( Q_{n,n+s}(t) f^{(n+s)} _{N,0}  -Q^0_{n,n+s}(t) f^{(n+s)} _{0} \right) (1- \indc _{\cB_{\eps_0}^{n-} })    \right\|_{L^\infty_{\beta',n}}\\
&  \qquad    \leq   C^{n+s} \left({  t\beta^{-(d+1)/2} \over \beta- \beta'} \right)^ s
\Big[ \Big( \exp (- (\beta-\beta') R^2/4) + (n+s)^2 {\delta \over t} +  \gamma(\eps, \eps_0)   \Big) \| f^{(n+s)} _{N,0}   \|_{L^\infty_{\beta,n+s}}\\
& \qquad\qquad\qquad\qquad\qquad \qquad +    \big\| \big(f^{(n+s)} _{N,0} - f^{(n+s)} _{0} \big)  (1- \indc _{\cB_{\eps_0}^{(n+s)-} })  \big\|_{L^\infty_{\beta,n+s}} \\
& \qquad \qquad\qquad\qquad\qquad\qquad+     \|    f^{(n+s)} _{0}  \exp (\beta H_{n+s} )  \| _{W^{1,\infty}_x(L^\infty_v)}  (n+s) \eps  \Big]\,,
\end{aligned} 
$$
with~$\beta'<\beta$.
\end{Prop}

\subsection{A refined convergence statement}

The previous argument shows that once   recollisions have been discarded, pseudo-trajectories are stable as $\eps \to 0$, in the sense that their distance to the corresponding Boltzmann pseudo-trajectory converges to 0.
 The only assumptions used to obtain the convergence of the marginals for times $t\in [0, t^*]$  are  that the initial data~$f_0$ has some regularity in space (the Lipschitz bound appearing in~ Proposition~\ref{error-est} could be weakened to H\"older continuity) and the initial marginals satisfy the uniform growth condition
 \begin{equation}
 \label{growth}
  \sup _N f_{N,0} ^{(n+s)} \leq C^{n+s}  \exp (-\beta H_{n+s})
  \end{equation}
together with  the convergence
$$\Big| (f_{N,0} ^{(n+s)} - f^{\otimes(n+s)}_0 ) (1- \indc _{\cB_{\eps_0} ^{(n+s)-}}) \Big| \leq     C^{n+s} \exp (-\beta H_{n+s}) \eps   \,.
$$
Actually any positive power of~$\eps$ would do in the above estimate.
Note   that not all configurations in $ \cD_\eps^{n+s} \setminus \cB_{\eps_0} ^{(n+s)-}$  are reached by the good pseudo-trajectories. Actually   a very small subset $V_{n+s,n}^+\subset \cD_\eps^{n+s} \setminus \cB_{\eps_0} ^{(n+s)-}$ of these configurations can be reached since one has the condition that, looking at the forward flow, if one particle disappears at each collision, we should end up with $n$ particles within a time $T$ (see Figure \ref{admissible-fig}). This imposes~$s$ conditions on the configuration $Z_{n+s}$. Note that, by definition, configurations of $V^+_{n+s,n}$   have at least one collision when evolved by the free flow ${\bf \Psi}_{n+s}$. Taking into account the additional constraint on the order of collisions, we can prove the following result.
\begin{Prop}\label{V-prop}
The set of admissible initial configurations (reached by pseudo-dynamics associated with the forward BBGKY hierarchy) satisfies
$$ | V_{n+s,n}^+ | \leq  (CR)^{s+n}(n+s-1)\dots n (RT \eps ^{(d-1)} )^s\,.$$
Furthermore,
$$ V_{n+s,n}^+ \subset  \cB_{\eps} ^{(n+s)+}\setminus  \cB_{\eps} ^{(n+s)-} \,.$$
\end{Prop}

\begin{figure}[h]
\begin{center}
\scalebox{0.4}{\includegraphics{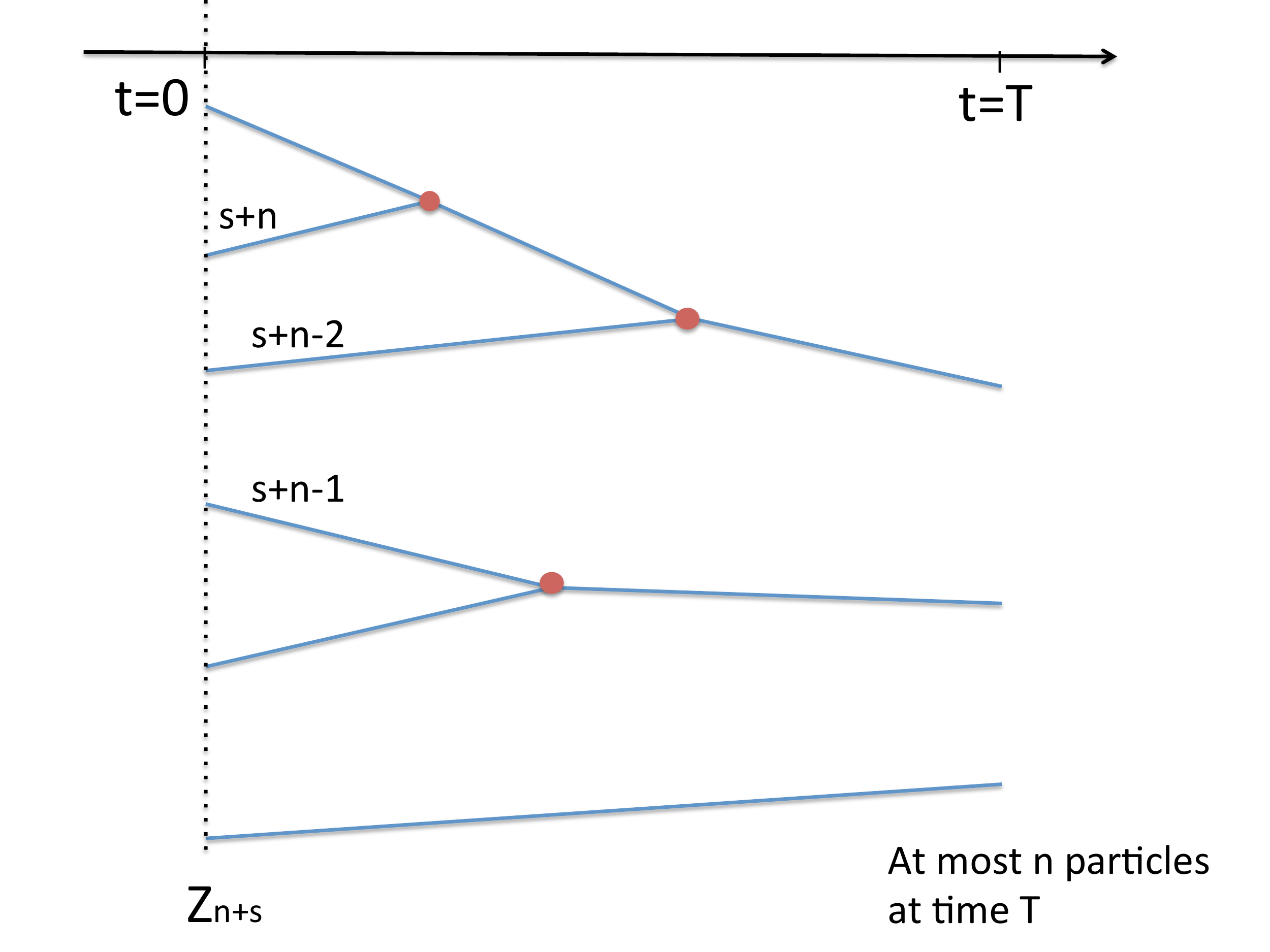}}
\caption{\label{admissible-fig} Admissible initial configurations.}
\end{center}
\end{figure}

We thus can state the following refined version of Lanford's theorem 
{\color{black} which provides quantitative convergence estimates outside the bad sets}.
\begin{Thm}
\label{lanford-thm2}
Consider a system of $N$ hard-spheres of diameter $\eps$ on $\T^d = [0,1]^d$ (with~$d\geq 2$), initially distributed according to some density $f_{N,0}$ satisfying the growth condition {\rm(\ref{growth})}
 for some $\beta>0$,
 together with the convergence
 \begin{equation}
 \label{init-cv}
\forall n \in [1,N] \, , \quad  \Big| (f_{N,0} ^{(n)} - f^{\otimes n}_0 ) \,  ( 1 - \indc_{ \cB_{\eps_0} ^{n-} }) \Big| 
\leq   C^{n} \exp (-\beta H_{n}) \eps^a\, , 
\end{equation}
 for some~$a>0$ and for~$|\log \eps | \eps \ll \eps_0 \ll 1$. Denote by $f$   the solution of the Boltzmann equation~{\rm(\ref{boltzmann})}.
Then,  in the Boltzmann-Grad limit~$N \to \infty$ with $N \eps^{d-1} = 1$, the marginal~$f_N^{(n)}$   converges  to $f^{\otimes n}$ uniformly on $(\cD_\eps^n \setminus \cB_{\eps_0} ^{n-} )\times [0,t^*] $, i.e.
there exists $\gamma' (\eps, \eps_0)$ converging to 0 such that uniformly in~$t \in [0,t^*]$, 
$$\Big| \big( f_{N} ^{(n)} (t) - f^{\otimes n} (t) \big) (1- \indc _{\cB_{\eps_0} ^{n-}})  \Big|     \leq  C^{n} \exp (-\beta' H_{n}) \gamma' (\eps, \eps_0) \,, 
$$ 
with~$  \beta'< \beta$ and  $t^*$ introduced in \eqref{eq: convergence}.
 \end{Thm}
Compared to \cite{BLLS}, this theorem provides a description of the geometry of the bad sets along the evolution, and quantitative estimates of their measures. Note that a similar notion of one-sided convergence has been introduced by Denlinger in \cite{denlinger}.

 \section{Irreversibility and time concatenation} \label{sec:ITC}

Note that the very same proof   shows that, in the Boltzmann-Grad limit, the marginal $f_N^{(n)}$   converges  to  $ \tilde f ^{\otimes n}$  where $\tilde f$ is the solution of the {\bf reverse} Boltzmann equation
\begin{equation}
\label{backboltzmann}
\left\{ \begin{aligned}
& \d_t \tilde f +v\cdot \nabla_x \tilde f= -    Q(\tilde f,\tilde f),\\
& Q(f,f)(v):=\iint_{{\mathbb S}^{d-1} \times \R^d} [f(v')f(v'_1)-f(v)f(v_1)]  \,
{\color{black} \big((v_1 -v)  \cdot \nu\big)_+ } \, dv_1 d\nu  \, ,\\
 \end{aligned}\right.
\end{equation}
uniformly on  $(\cD_\eps^n \setminus \cB_{\eps_0} ^{n+} )$ and times from~$0$ to~$- t^* $. 
This convergence requires only the growth condition~(\ref{growth}) and the initial convergence
 \begin{equation}
 \label{init-cv2}
 \Big| (f_{N,0} ^{(n)} - f^{(n)}_0 )
 {(1 - 1_{  \cB_{\eps_0} ^{n+}}) }
 \Big| \leq   C^{n} \exp (-\beta H_{n}) \eps^\gamma ,
 \end{equation}
 for some~$\gamma>0$.

We thus have a {\bf symmetric situation for negative and positive times}, which indicates once more that the initial data play a very special role distinguishing between the direct and reverse Boltzmann dynamics.


\subsection{Irreversibility}\label{irreversibility-section}
\subsubsection{At the macroscopic level}

Recall that the Boltzmann dynamics admits a Lyapunov functional. Indeed,
using the well-known facts (see \cite{CIP})  that the mappings~$(v,v_1)\mapsto(v_1,v)$ (microscopic exchangeability) and~$(v,v_1,\nu)\mapsto \left(v',v_1',\nu\right)$ (microscopic reversibility)   have unit Jacobian determinants and preserve the cross-section,   one can show that formally for any test function~$\varphi$
\begin{equation}
\label{symmetry-boltz}
	\int Q(f,f)\varphi dv=\frac14 \iiint [f'f'_1-ff_1] (\varphi+\varphi_1 - \varphi'-\varphi'_1) ((v_1-v)\cdot \nu)_+ \, dvdv_1 d\nu\,,
\end{equation}
with the short notation $ f'= f(v'),  f'_1 = f(v'_1), f_1 = f(v_1)$,  and similarly for~$\varphi$.

Disregarding integrability issues, we choose~$\varphi=\log f$  in (\ref{symmetry-boltz}), and use the properties of the logarithm, to find
\begin{equation}
\label{dissipation-def}
\begin{aligned}
D(f)& \equiv  -\int Q(f,f) \log f dv\\
& = \frac14 \int_{\R^d\times \R^d \times {\mathbb S}^{d-1}}  (f'f'_1-ff_1) \log {f'f'_1\over ff_1} \, ((v-v_1)\cdot \nu)_+ \, dvdv_1 d\nu \geq 0 \, .
\end{aligned}
\end{equation}
The so-defined entropy production is therefore a nonnegative functional {\color{black} in agreement with the second principle of thermodynamics}.

This leads to  Boltzmann's H-theorem,  stating that the entropy is (at least  formally) a Lyapunov functional for the Boltzmann equation.
\begin{Prop}
Let $f=f(t,x,v)$ be a smooth solution to the Boltzmann equation {\rm(\ref{boltzmann})} with initial data $f_0$ of finite relative entropy with respect to some Gaussian (equilibrium) distribution $M=M(v)$
$$\int f_0\log {f_0 \over M} dv dx <+\infty\,.$$
Then, for all $t\geq 0$
\begin{equation}
\label{H-thm}
 \int f\log {f \over M} (t,x,v) dv dx+ \int_0^t \int D(f) (\tau,x) dxd\tau \leq   \int f_0\log {f_0 \over M} dv dx\, .
\end{equation}
\end{Prop}

The classical interpretation of the H-theorem is that entropy measures the  quantity of microscopic information that is known on  the system. 
{\bf Irreversibility is related to a loss of information in our description of the dynamics}.

Note that, for negative times, the distribution is evolved according to the reverse Boltzmann dynamics, and we have
$$ \int f\log {f \over M} (-t,x,v) dv dx- \int_0^{-t }\int D(f) (\tau ,x) dxd\tau \leq   \int f_0\log {f_0 \over M} dv dx\, ,
$$
so that the global picture for the entropy should look like Figure \ref{entropy-fig}.
\begin{figure}[h]
\begin{center}
\scalebox{0.4}{\includegraphics{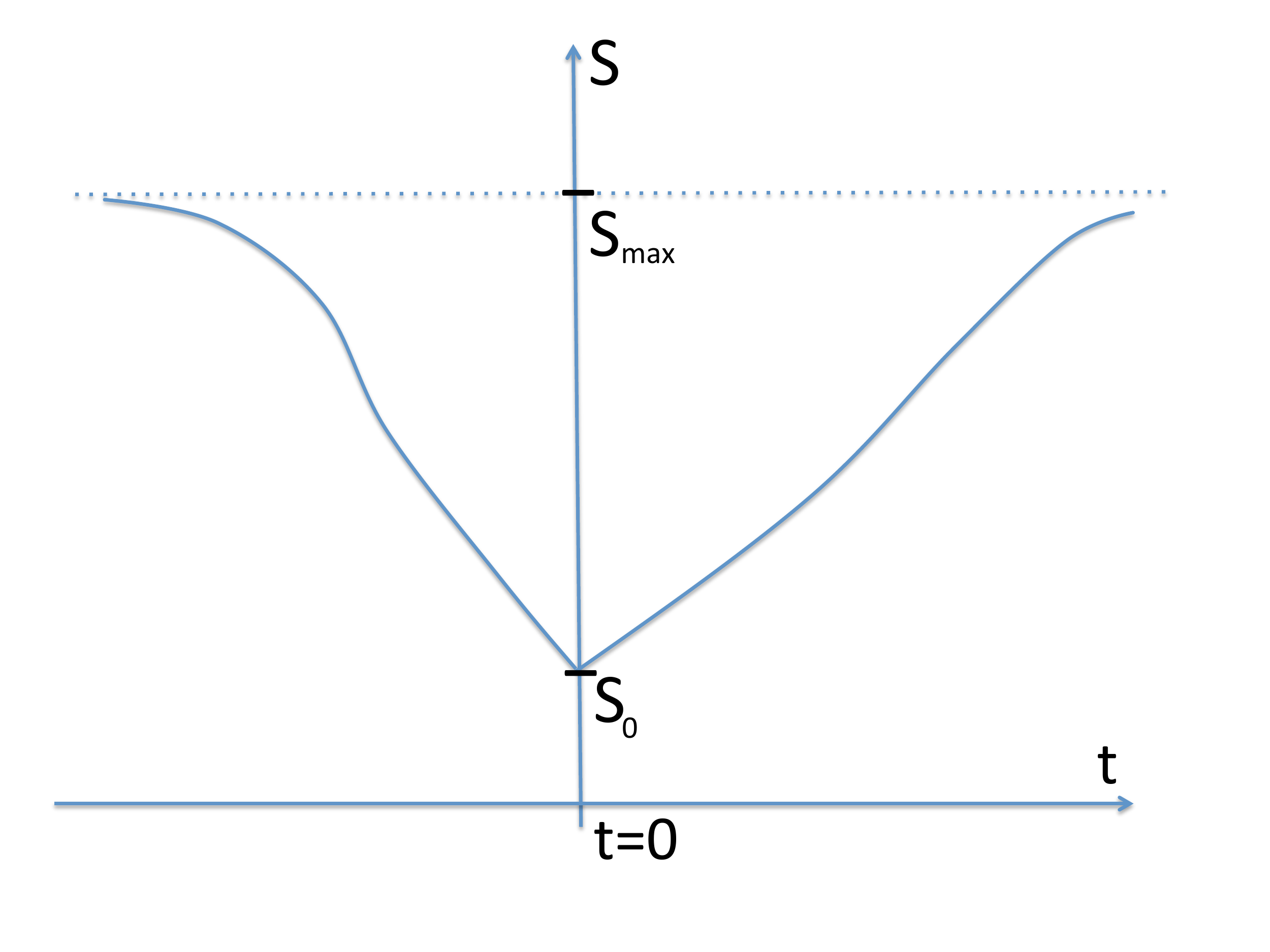}}
\caption{\label{entropy-fig}Variations of the entropy $S$.}
\end{center}
\end{figure}

\bigskip

It is important to realize that the {\bf loss of reversibility 
is already present at the level of the Boltzmann hierarchy} and does not come from some  averaging or projection in   phase space.
In particular, it has nothing to do with the chaos assumption. Indeed, it can be shown  that the Boltzmann hierarchy is irreversible: from the Hewitt-Savage theorem (see~\cite{hewitt-savage}) and the symmetry assumption on the labels, we indeed know that the initial data can be decomposed as a superposition of chaotic initial data, i.e.  that there exists a measure~$\pi$ on the space of probability densities such that 
$$ 
f_{0} ^{(n)}= \int g_0^{\otimes n} d\pi (g_0) \hbox{ for any } n\in \N^*\,.
$$

Then, by linearity of the Boltzmann hierarchy (\ref{hierarchyboltzmann}), we deduce that the family $(f^{(n)} (t))_{n\in \N^*}$ defined by
$$ f^{(n)}(t) = \int (g(t)) ^{\otimes n} d\pi (g_0)$$
where $g(t)$ is the solution to the Boltzmann equation with initial data $g_0$, is a solution to the Boltzmann hierarchy.
Since the entropy is nondecreasing for each solution of the Boltzmann equation,  we deduce that 
$$ S(t) :=-\int  (g(t) \log g(t) ) d\pi(g_0)  $$
 is nondecreasing, which encodes irreversibility.

This result 
means that microscopic information has been lost in the limiting process.
 
\subsubsection{At the microscopic level}

Let us now consider an intermediate time $\tau$, positive but strictly smaller than Lanford's time $t^* $ in Theorem \ref{lanford-thm}.
We would like to  reverse time and look at the convergence of the BBGKY hierarchy on $[\tau',\tau]$ for~$\tau'<\tau$, starting from time $\tau$. 

\begin{figure}[h]
\begin{center}
\scalebox{0.4}{\includegraphics{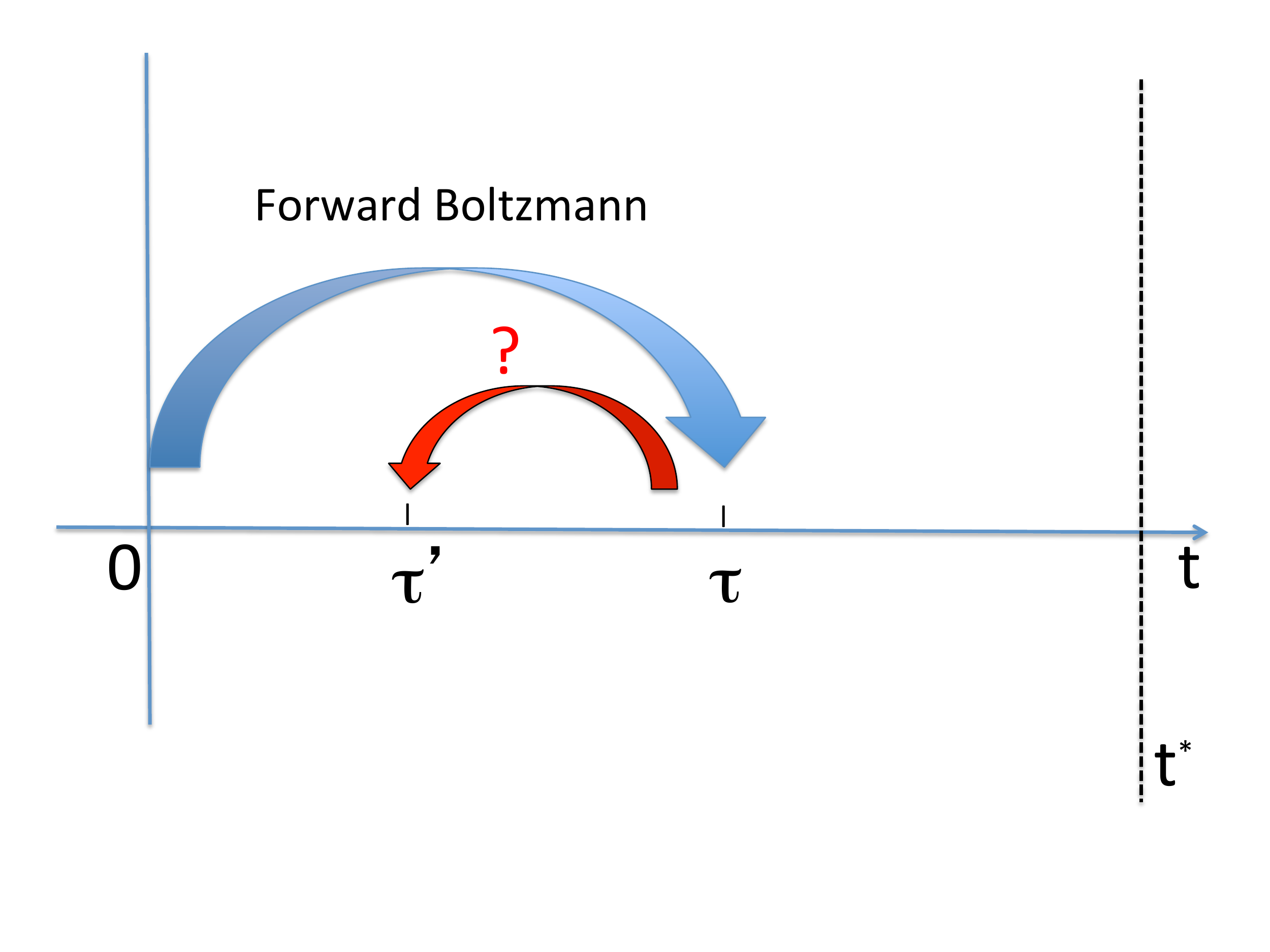}}
\caption{Irreversibility. The measure at time zero leads to the Boltzmann equation, but it is not possible to apply Lanford's theorem on $[\tau',\tau]$
taking the measure at time $\tau$ as initial condition.}
\end{center}
\end{figure}

Assume that at time 0 the data  $f_{N,0}$ is almost factorized, in the sense of~(\ref{eq: initial data 2}). We have seen that, for any fixed~$n$, as $\eps \to 0$, the marginals $(f_N^{(n)}) $ converge uniformly to $f^{\otimes n}$  outside from $\cB_{\eps_0}^{n-}$, where~$f$ solves the Boltzmann equation~:
$$ \|  ( f_N^{(n)} (t)  - f^{\otimes n}(t) )(1- \indc_{ \cB_{\eps_0}^{n-}})  \|_{\infty } \to 0 \hbox{ as } \eps \to 0\,.$$
By definition, starting from $\tau$,  $f$ is a solution of the {\bf backward} Boltzmann equation on~$[0,\tau]$ (namely of the problem at final values).
This would obviously not be the case if we were imposing  chaotic data at time $\tau$ 
{\color{black} similar to the one at time $0$.}
Indeed, the entropy would be decreased by the reverse Boltzmann equation while it is increased along the backward Boltzmann dynamics.
This means therefore that {\bf the structure of the family~$\big(f_N^{(n)} (\tau)\big)_{n\leq N}$ is very different from the chaotic structure of the initial data}.

\bigskip

Now let us turn to the question of irreversibility. We fix two times~$0 < \tau' < \tau < t^*$. Consider the representation formula~(\ref{duhamel1}) for the marginals $f_N^{(n)} $   
$$ \begin{aligned}
 f^{(n)} _N(\tau') =\sum_{s=0}^{N-n} \int_0^{\tau'} \int_0^{t_{n+1}}\dots  \int_0^{t_{n+s-1}}  {\bf S}_n(\tau'-t_{n+1}) C_{n,n+1}  {\bf S}_{n+1}(t_{n+1}-t_{n+2}) C_{n+1,n+2}   \\
\dots  {\bf S}_{n+s}(t_{n+s})     f^{(n+s)}_{N,0} \: dt_{n+s} \dots d t_{n+1} \, .
\end{aligned}
$$
It can be written starting from time~$\tau$ instead of $0$, meaning
$$ 
\begin{aligned}
 f^{(n)} _N(\tau') =\sum_{s=0}^{N-n} 
\int_\tau^{\tau'}  \int_\tau^{t_{n+1}}\dots  \int_\tau^{t_{n+s-1}}  {\bf S}_n(\tau'-t_{n+1}) C_{n,n+1}  {\bf S}_{n+1}(t_{n+1}-t_{n+2}) C_{n+1,n+2}   \\
\dots  {\bf S}_{n+s}(t_{n+s}-\tau)     f^{(n+s)}_{N}(\tau) \: dt_{n+s} \dots d t_{n+1} \, ,
\end{aligned}
$$
 since the Liouville equation~(\ref{Liouville}) satisfied  by~$f_N$  is   reversible and autonomous with respect to time (it generates a group of evolution).
As usual for analytic functions, the radius of convergence of the series at $\tau$ is at least $t^*-\tau$. 

What we would need  to apply the refined version of Lanford's theorem (Theorem \ref{lanford-thm2}) starting from time~$\tau$ and moving back to~$\tau'$ is the convergence of $f_N^{(n+s )}(\tau) $ on the sets $ V_{n+s,n}^-$ which consist of  the  configurations  of $n+s$ particles at time $\tau$ reached by good  pseudo-dynamics having $s$ collisions on $[\tau',\tau]$. Note that these pseudo-trajectories are built forward {\color{black} as they go from time $\tau'$ to $\tau$} and that we have
$$ 
V_{n+s,n}^- \subset  \cB_{\eps} ^{(n+s)-}\setminus  \cB_{\eps} ^{(n+s)+} \,,
$$
which is the symmetric counterpart of Proposition \ref{V-prop}.

Recollisions of the backward dynamics are indeed exactly collisions of the forward pseudo-dynamics.  This implies that we have no information about the convergence of $f_N^{(n)}(\tau) $ on the sets $V_{n+s,n}^-$, and that {\bf we cannot prove the convergence to the reverse Boltzmann dynamics on $[\tau',\tau]$ starting from $\tau$} (which is consistent with the fact that the reverse Boltzmann dynamics is not the backward Boltzmann dynamics!).
For the same reasons the argument behind the so-called Loschmidt's paradox fails. Indeed if at time $\tau$ we invert all the velocities and consider $f_N^{(n)}(\tau,X_n,-V_n)$ as initial data, we cannot apply Theorem \ref{lanford-thm2} so that there is no contradiction with the backtracking of marginals.
{\color{black} 
The same argument was already put forward in~\cite{BLLS}. 
}

\begin{Rmk} \label{rem:revbaddata}
Evolving a chaotic data by the reverse Boltzmann dynamics gives  a systematic method to construct data for which the Boltzmann-Grad limit fails to hold, even though we do have a weak chaos property  in the sense of Definition~{\rm\ref{def: chaos}}.
In Section~{\rm\ref{slightmodif}}, we   show a more explicit construction leading to an almost chaotic initial data, with modifications of the second order correlations on a small set,
such that the limiting dynamics is free transport (far from the Boltzmann dynamics). 
\end{Rmk}

\subsection{Time-concatenation}\label{iteration-sec}

Another important feature of the limiting equation is that  one can iterate in the sense of the following Proposition.

\begin{Prop}
Let $f$ be a smooth solution of  the Boltzmann equation~{\rm(\ref{boltzmann})}  on $ [0,\tau ] $ with initial data $f_0$, and assume there is a smooth solution $f$ of the Boltzmann equation on $[\tau, t^*]$ with data $f(\tau)$ at $\tau$.
Then, $f$ is the same solution of the Boltzmann equation on $[0, t^*]$.
\end{Prop}

This property is a simple consequence of the fact that the Boltzmann equation is a local in time partial differential equation, with no memory effect. 
It is a kind of Markov property of the underlying process.

\begin{figure}[h]
\begin{center}
\scalebox{0.4}{\includegraphics{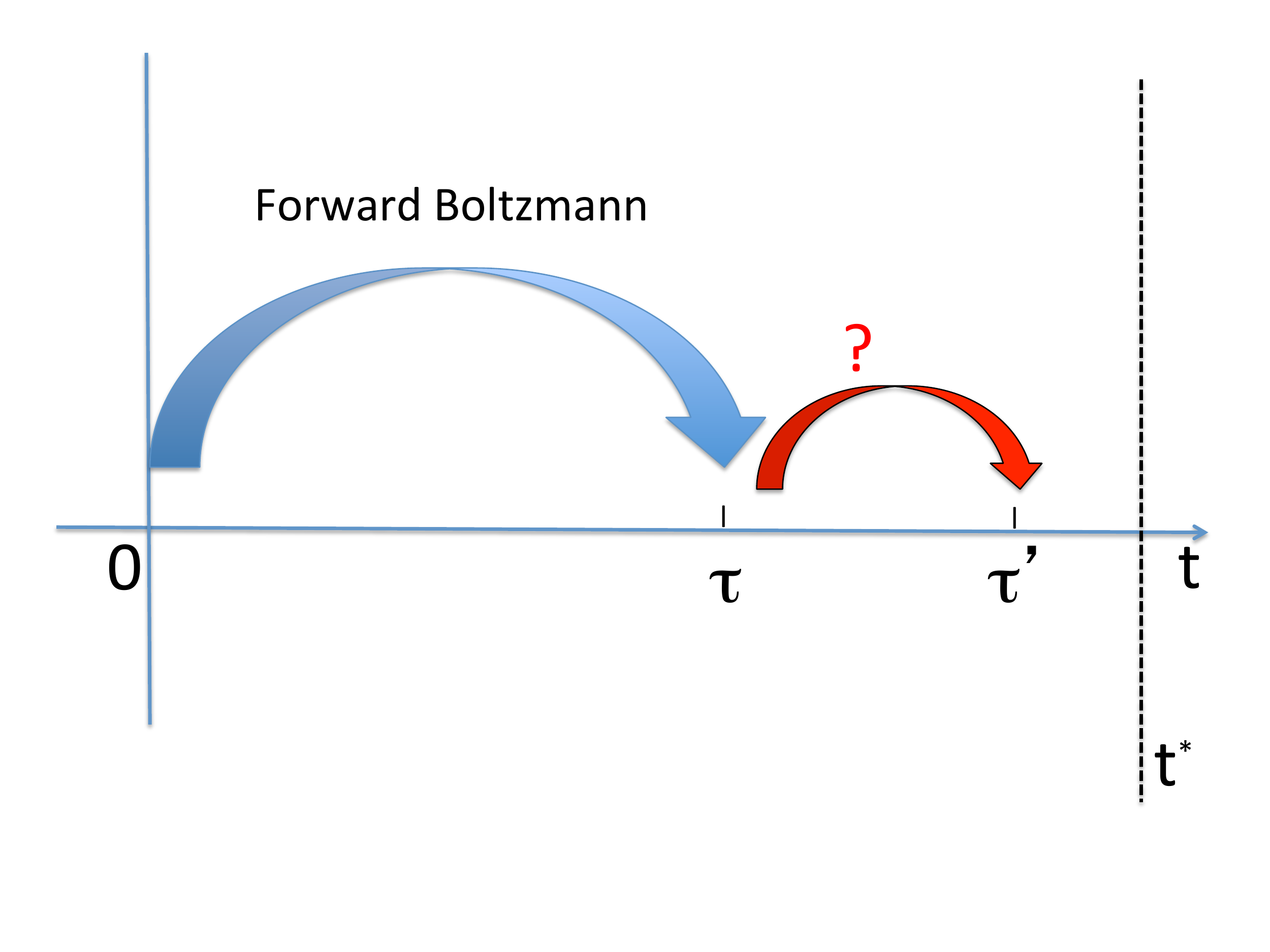}}
\caption{Time-concatenation. The measure at time zero leads to the Boltzmann equation on $[0,\tau]$ and it is possible to re-apply Theorem \ref{lanford-thm2} 
on $[\tau, \tau'] $ taking the measure at $\tau$ as initial condition.}
\end{center}
\end{figure}

\bigskip

Let  $\tau<\tau'$ denote   intermediate times, positive but strictly smaller than Lanford's time~$t^*$. 
As previously, we denote  by $f_N$  the solution to the Liouville equation with chaotic initial data in the sense of Theorem \ref{lanford-thm}.
If we want to iterate Lanford's convergence proof on~$[\tau, \tau']$, what we need (in addition to the uniform $L^\infty$ a priori estimate) is the convergence of~$f_N^{(n+s) }(\tau)$ on the sets $V_{n+s, n}^+$ reached by  good (backward) pseudo-trajectories. By definition, we  have~$V_{n+s,n}^+\subset \cD_\eps^{n+s} \setminus \cB_{\eps_0} ^{(n+s)-}$.

And from the refined version of Lanford's theorem (Theorem \ref{lanford-thm2}), we have that
$$ \|  \big( f_N^{(n)} (\tau) -f^{\otimes n} (\tau) \big) (1- \indc_{ \cB_{\eps_0}^{n-}})  \|_{\infty } \to 0 
\quad \hbox{ as } \eps \to 0\,.$$

Combining both properties, we deduce that we can iterate the convergence as long as the growth condition (\ref{growth}) is satisfied.

\begin{Rmk}
 Note that the main limiting factor to  extend the convergence time  is the loss with respect to $\beta$ in the estimate~{\rm(\ref{error-est})}. 
 The previous iteration argument fails therefore to improve  the time of convergence in Lanford's theorem for initial data of the form {\rm(\ref{eq: initial data 2})}.
 
 For initial data close to equilibrium,  it is   proved in  \cite{BGSR1, BGSR2} that one can  actually reach times of the order~$O(\log\log \log N)$.  The proof relies on global a priori bounds, it consists in designing a 
 subtle pruning procedure to get rid of the contribution of super-exponential collision trees and then to express the contribution of all other dynamics in terms of the initial data.

%
 \end{Rmk}

\section{Chaotic initial data leading to different dynamics}
\label{slightmodif}

At large scales, the propagation of chaos \eqref{eq: Stosszahlansatz} holds and the measure factorizes,  but the memory of the Hamiltonian dynamics remains encoded in $f_N(t)$ on very specific configuration sets of size vanishing with~$\eps$.
We are not able to describe the refined structure of the correlations in the density $f_N(t)$, but we are going to introduce an example which illustrates how constraints on very small sets may change the nature of the dissipative dynamics. Unlike the one obtained by reversing velocities (see Remark~\ref{rem:revbaddata}), this example will be totally explicit.

Using the notation  \eqref{eq: bad set} of the bad sets,    
we consider the initial data
\begin{equation}
\label{eq: initial data modified}
\hat f_{N,0} (Z_N ) := \frac1{\hat \cZ_N} \prod_{i=1}^Nf_0(x_i, v_i)   \,  
\indc_{ \{ Z_N \notin \cB_{\eps}^{N +} \}} \, ,
\end{equation}
where $\cB_{\eps}^{N +}$ is the set such that some collision occurs between the $N$ particles within a time~$T$. Contrary to the definition \eqref{eq: bad set}, we   choose $T$ as a short time and set $T = \delta >0$.
By construction the measure $\hat f_{N,0}$  will evolve according to free transport on the time interval~Â$[0, \delta]$ as there are no interactions between the particles. In particular, the evolution of the first marginal ${\hat f}_{N,0}^{(1)}$ is no longer approximated by the Boltzmann equation in the time interval~$[0, \delta]$  and there is no dissipation, even at the level of the marginals.

In the following, we are even going to argue that, at a macroscopic scale, the structure of the measure 
\eqref{eq: initial data modified} behaves essentially as the one of the initial data $f_{N,0}$ given in~\eqref{eq: initial data 2} for which Lanford's Theorem holds. In particular, we deduce that a chaos property \eqref{eq: Stosszahlansatz} holds for
the measure~$\hat f_{N,0}$. The key point is that the two measures differ on very singular sets which are exactly the relevant sets for the microscopic evolution.

To prove this, it is convenient to rephrase the measure \eqref{eq: initial data modified}, which has a fixed number of particles, in a slightly different setting where $N$ is varying. The terminology ``canonical'' and ``grand canonical'' ensemble (inherited from statistical physics) is used, respectively, for the two pictures. In the new setting one introduces ``rescaled correlation functions'' $f_{\eps,0}^{(j)}$ describing the same macroscopic behaviour as the marginals ${\hat f}_{N,0}^{(j)}$. For our present purpose the $f_{\eps,0}^{(j)}$ have some remarkable advantage, as they can be dealt with by using methods of expansion developed in different contexts \cite{Ko06,Poghosyan Ueltschi} (for complications of the cluster expansion techniques due to a canonical formulation, see \cite{PT12}).

\medskip
\subsection{The grand canonical formalism}

The grand canonical phase space is 
$$
{\mathcal D}_\eps = \cup_{n\geq 0} {\mathcal D}_\eps^n
$$
(actually ${\mathcal D}_\eps^n = \emptyset$ for $n$ large, due to the exclusion).
Given~$(f_{n,0} )_{n\geq 0}
$ we assign the collection of probability densities for the configuration $Z_n \in  {\mathcal D}_\eps^n$, $n=0,1,\dots$:
$$\frac{1}{n!}W_{\eps,0}^n(Z_n) 
:= \frac{1}{\cZ_\eps} \frac{\mu_\eps^n}{n!} f_{n,0} (Z_n) \, ,
$$
where $\mu_\eps = \eps^{-d+1}$. The normalization constant $\cZ_\eps$ is given by 
$$\cZ_\eps :=  \sum_{n\geq 0}\frac{\mu_\eps^n}{n!} \hat \cZ_n
\quad \text{with} \quad 
\hat \cZ_n := \int dZ_n f_{n,0} \,.
$$
$\{W_{\eps,0}^n\}_{n\geq 0}$ defines the grand canonical state on ${\mathcal D}_\eps$, 
normalized as $$\sum_{n\geq 0}\frac{1}{n!}\int W_{\eps,0}^n(Z_n) dZ_n =1\, .$$

The total number of particles $\cN$ is random and distributed according to 
\begin{equation*}
{\mathbb P}_{\mu_\eps} \big( \cN = n \big)  =  \frac{1}{\cZ_\eps} \frac{\mu_\eps^n}{n!}
\hat \cZ_n \, .
\end{equation*}
The choice $\mu_\eps = \eps^{-d+1}$ ensures that the average number of particles grows as $\eps^{-d+1}$,
hence the inverse mean free path remains of order $1$ (Boltzmann-Grad scaling)
\begin{equation}
\label{eq: Boltzmann Grad}
\lim_{\e\to 0}{\mathbb E}_{\mu_\eps} \big( \cN \big)   \eps^{d-1} = \kappa > 0\, .
\end{equation}
We postpone this check to the end of the section.

%
%

\bigskip

Let us define the $j$-particle correlation function, $j=1,2,\dots$.
The idea is to count how many groups of $j$ particles fall, in average, in a given configuration $Z_j = (z_1,\dots,z_j)$:
\begin{align*}
\rho_{\e,0}^{(j)} (z_1,\dots,z_j)=\Big\langle \sum_{k_i\neq   k_j} \delta_{\zeta_{k_1}}(z_1)\dots 
\delta_{\zeta_{k_j}}(z_j)\Big\rangle \, ,
\end{align*}
where we are labelling the particles and indicating their (random) configuration by $\zeta_1,\cdots,\zeta_n$, and
the brackets denote average with respect to the grand canonical state.
In terms of the densities it is
\begin{align*}
\rho_{\eps,0}^{(j)} (Z_j)=\sum_{n=0}^{\infty} \frac{1}{n!} \int dz_{j+1}\cdots dz_{j+n} 
W^{j+n}_{\eps,0}(Z_{j+n})\, .
\end{align*}

In the case with minimal correlations, i.e. when 
\begin{equation}
f_{n,0}:=\prod_{i=1}^n f_0(x_i, v_i)  \prod_{i\neq j} \indc_{|x_i-x_j|>\eps }\;,
\label{eq:GCfn0}
\end{equation} 
one has
\begin{align}
f_{\eps,0}^{(j)} (Z_j) & := \mu_\eps^{-j}\rho_{\eps,0}^{(j)}(Z_j)\nonumber\\
& =\Big[f^{\otimes j}_0 \prod_{1\leq i<k \leq j} \indc_{|x_i- x_k|Ê>\eps}  \Big]
\nonumber  \\
&  \times \frac{1}{\cZ_\eps}\sum_{n \geq 0}\frac{\mu_\eps^n}{n!} 
\int dz_{j+1}\cdots dz_{j+n} 
f^{\otimes n}_0 \left(\prod_{i=1}^j \prod_{k=j+1}^{j+n} \indc_{|x_i- x_k|Ê>\eps}\right)
\left(\,\prod_{j+1\leq i < k \leq {j+n}} \indc_{|x_i- x_k|Ê>\eps}\right)  \nonumber \\
& \leq f^{\otimes j}_0  \indc_{ \{ Z_j \notin \cB_{\eps}^{j +} \}}  (Z_j) \, ,
\label{eq: uniform grand canonique}
\end{align}
where the last inequality follows by removing the constraint between the $j$ particles and the rest of the system.
Note that the rescaled correlation functions $f_{\eps,0}^{(j)}$ are quantities of order~1 in~$\eps$.

\bigskip

The Boltzmann equation can be derived for both ensembles \cite{BLLS,S2,PS}. 
\begin{Thm}[\cite{BLLS}]
Consider a system of  hard-spheres of diameter $\eps$ on the~$d$-dimensional periodic box $\T^d = [0,1]^d$ (with~$d\geq 2$), initially in  the grand canonical state
with~$f_{n, 0}$ given by~{\rm(\ref{eq:GCfn0})} and~$f_0$ satisfying~{\rm(\ref{eq: initial data 1})}.

Then, as $\eps \to 0$, the rescaled correlation function $f_\eps^{(1)}$   converges almost everywhere to the solution of the Boltzmann equation~{\rm(\ref{boltzmann})}
with initial data~$f_0$, {\color{black} on a time interval $[0, t^*]$ where~$t^*$ depends only on the parameters $\beta,\mu$ of~\eqref{eq: initial data 1}}.
\end{Thm}

\subsection{A counterexample}
A natural reformulation of \eqref{eq: initial data modified} with varying number of particles is obtained as follows.  Define
\begin{equation}
\label{eq: grand canonique}
\frac{1}{n!}W_{\eps,0}^n(Z_n) 
:= \frac{1}{\cZ_\eps} \frac{\mu_\eps^n}{n!} f^{\otimes n}_0(Z_n)
\indc_{ \{ Z_n \notin \cB_{\eps}^{n +} \}}
= \frac{1}{\cZ_\eps} \frac{\mu_\eps^n}{n!} f^{\otimes n}_0(Z_n)  \prod_{i<j}\left(1+\zeta_{ij}\right),
\end{equation}
where $\mu_\eps = \eps^{-d+1}$ and $\zeta_{ij} =\zeta(z_i,z_j) =   - \indc_{\cC}(z_i,z_j)$ with  
$\cC$ the set leading to a collision
$$
\cC := \Big \{ (z_i,z_j) \in (\T^{d} \times \R^{d})^2, \quad \exists s \in [0,\delta], 
\quad  \big| x_i - x_j + s (v_i - v_j) \big| \leq \eps \Big\}.
$$
The normalization constant $\cZ_\eps$ is given as above by 
\begin{equation*}
\cZ_\eps :=  \sum_{n\geq 0}\frac{\mu_\eps^n}{n!} \hat \cZ_n
\quad \text{with} \quad 
\hat \cZ_n := \int dZ_n f_0^{\otimes n} \prod_{i<j}\left(1+\zeta_{ij}\right)\,.
\end{equation*}

By construction, the  grand canonical density \eqref{eq: grand canonique} evolves according to the free transport dynamics in the time interval $[0,\delta]$,
\begin{equation}
\label{eq: free evolution}
\forall t \leq \delta, \qquad 
f_{\eps}^{(j)} (t,Z_j) := \mu_\eps^{-j}\rho_{\eps}^{(j)}(t,Z_j) = {\bf S}^0_j (t) f_{\eps,0}^{(j)}(Z_j)\, .
\end{equation}
The rescaled correlation functions $f_{\eps,0}^{(j)}$ obey some of the assumptions required to apply Lanford's theorem, in particular the key  $L^\infty$ bound holds thanks to \eqref{eq: uniform grand canonique}. Moreover, we will see in Proposition~\ref{prop: factorisation} below that a chaos property holds in a sense stronger than \eqref{eq: Stosszahlansatz}.
Nevertheless the correlation functions are irregular at the microscopic scale on the sets $ \cB_{\eps}^{j +}$ so that 
Lanford's proof cannot apply and there is no contradiction with \eqref{eq: free evolution}.
Note that the constraints are imposed only in the forward direction, thus we expect to get the reverse Boltzmann equation for negative times.

\medskip

To conclude this example, we will show that the state is chaotic.
\begin{Prop}
\label{prop: factorisation}
The measure $\{W_{\eps,0}^n\}_{n\geq 0}$ is asymptotically chaotic,
uniformly outside a bad set of configurations in $\cD_\eps$. More precisely, there exists $f^{(1)}:\T^{d} \times \R^{d} \to \R^+$
such that
\begin{equation}
\label{eq: factorisation hors cC} 
\begin{aligned}
&  \lim_{\eps \to 0} \,\sup_{ z }\,\Big|f_{\eps,0}^{(1)}(z) - f^{(1)}(z) \Big|=0\;,\\
& \lim_{\eps \to 0} \,\sup_{ Z_j \not \in \cB_{\eps | \log \eps|}^{j +}}
\,\Big| f_{\eps,0}^{(j)} (Z_j) - f_{\eps,0}^{(1)} (z_1)\cdots f_{\eps,0}^{(1)} (z_j) \Big|
= 0\,  ,
 \end{aligned}
\end{equation}
for all $j \geq 2$.
\end{Prop}

The result  for $j=2$ will follow by applying Theorem 2.3 of \cite{Poghosyan Ueltschi} (recalled below) where the decay of correlations has been estimated by means of cluster expansion.
\begin{Thm}[\cite{Poghosyan Ueltschi}]\label{Poghosyan Ueltschi}
Assume that there exist non negative functions $a$ and $b$ such that
\begin{equation}
\label{PU- assumptions}
\begin{aligned}
\forall n, \quad \forall (z_1,\dots, z_n) \, , \quad \prod _{1\leq i<j\leq n} (1+\zeta_{ij}) \leq \prod_{i=1}^n e^{b(z_i)}\,,\\
\forall z_i \, , \quad 
{\eps^{1-d}}
\int f_{0}( z_j) |\zeta_{ij}| e^{a(z_j) +2b(z_j)} dz_j \leq a(z_i)\,.
\end{aligned}
\end{equation}
Then, for almost all $z_1, z_2$,
\begin{align*}
& \Big| f_{\eps,0}^{(2)} (z_1,z_2) - f_{\eps,0}^{(1)} (z_1) f_{\eps,0}^{(1)} (z_2) \Big|
\leq   f_0(z_1) f_0(z_2) e^{a(z_1) +a(z_2) +2b(z_1) +2b(z_2)} \\
& \ \ \ \ \ \ \ \ \ \ \cdot\Big\{|\zeta_{12}| +
 \sum_{m \geq 1} \mu_\eps^m \int dZ'_m f_0^{\otimes m}(Z'_m) |\zeta(z_1,z'_1)\zeta(z'_1,z'_2)\cdots \zeta(z'_m,z_2)|
 \prod e^{a(z'_i) + 2b(z'_i) }\Big\}\, .
\end{align*}
\end{Thm}

\begin{proof}[Proof of Proposition {\rm\ref{prop: factorisation}} when $j=2$]
Assumptions~(\ref{PU- assumptions}) of Theorem~\ref{Poghosyan Ueltschi} hold by choosing $b=0$ and
$a(v) = c \delta  ( 1+|v|) $ with $\delta$ small enough, for some constant $c$ (depending on $\b,\mu, d$ of~\eqref{eq: initial data 1}).
As a consequence, Theorem~\ref{Poghosyan Ueltschi} leads to 
\begin{align*}
& \Big| f_{\eps,0}^{(2)} (z_1,z_2) - f_{\eps,0}^{(1)} (z_1) f_{\eps,0}^{(1)} (z_2) \Big|
\leq  f_0^{\otimes 2} e^{c \delta  (2+|v_1|+|v_2|)} \\
& \ \ \ \ \ \ \ \ \ \ \cdot\Big\{|\zeta_{12}| +
 \sum_{m \geq 1} \mu_\eps^m \int dZ'_m f_0^{\otimes m}(Z'_m) |\zeta(z_1,z'_1)\zeta(z'_1,z'_2)\cdots \zeta(z'_m,z_2)|
 e^{c \delta m + c \delta \sum_{i=1}^m |v'_i|}\Big\}\, .
\end{align*}
For $\delta$ small, the prefactor is bounded by $c \, e^{-\frac{\b}{4}v^2}$ as $f_0$ satisfies \eqref{eq: initial data 1}.
Moreover, for $(z_1,z_2)$ outside $\cB_{\eps | \log \eps|}^{2 +}$, the first term  $\zeta_{12}$ is equal to 0.
Then the proof of \eqref{eq: factorisation hors cC} boils down to showing that 
\begin{align}
\label{eq: decay}
\lim_{\eps \to 0} \sum_{m \geq 1} c^m\eps^{m(1-d)} \int dZ'_m  |\zeta(z_1,z'_1)\zeta(z'_1,z'_2)\cdots \zeta(z'_m,z_2)|
 e^{ - \frac{\beta}{4} \left(v_1^2 + v_2^2 + \sum_{i=1}^m |v'_i|^2\right)} = 0\,  
\end{align}
uniformly out of $\cB_{\eps | \log \eps|}^{2 +}$.
Given a velocity $v$, we define a cylinder associated with $z_1= (x_1,v_1)$ by
$$
\cR (z_1, v):= \Big\{ x \in \T^d, \qquad \exists s \in [0,\delta], \quad \big| x - x_1 +  s (v-v_1) \big| \leq \eps \Big\}\,.
$$
The measure of $\cR (z_1, v)$ is of order $\eps^{d-1} \delta |v - v_1|$.

\medskip

We first treat the term $m=1$ and show that for some constant $C>0$
\begin{equation}
\label{eq: m=1}
\eps^{1-d} \int dz'_1 |\zeta(z_1,z'_1) \zeta(z'_1,z_2)|
 e^{ - \frac{\beta}{4} \left(v_1^2+v_2^2 +|v'_1|^2\right)}  \leq  \frac{C}{ | \log \eps |^{1/2}}  \,  \cdotp
\end{equation}
Given $z_1,z_2$, we distinguish two cases to evaluate the measure of the overlap 
$\cR (z_1, v_1') \cap \cR (z_2, v_1')$. 
Let $\alpha$ be the angle between the axis of both cylinders, i.e.\, the angle between~$v'_1 -v_1$ and $v'_1 -v_2$.


\smallskip

\noindent
$\bullet$ {\it Suppose that $| \sin (\alpha) | \geq \eps | \log \eps |^{1/2}$}, then the angle between both axis is large enough so that the overlap $\cR (z_1, v_1') \cap \cR (z_2, v_1')$ has a volume of order at most $\eps^{d-1} /{| \log \eps |^{1/2}}$.
We get 
\begin{equation}
\label{eq: m=1 part1}
\eps^{1-d} \int dz'_1 \indc_{\{ | \sin (\alpha) | \geq \eps | \log \eps |^{1/2} \}}  
|\zeta(z_1,z'_1) \zeta(z'_1,z_2)|
e^{ - \frac{\beta}{4} \left(v_1^2+v_2^2 +|v'_1|^2\right)}  \leq  \frac{C}{| \log \eps |^{1/2}}  \,
\cdotp
\end{equation}

\smallskip

\noindent
$\bullet$ {\it Suppose that  $| \sin(\alpha) |\leq \eps | \log \eps |^{1/2}$}, then the cylinders $\cR (z_1, v_1')$ and $\cR (z_2, v_1')$ are almost parallel and they are anchored at $x_1, x_2$. Recall that $(z_1,z_2)$ is outside $\cB_{\eps | \log \eps|}^{2 +}$ so that $|x_1 - x_2| \geq \eps  | \log \eps|$. The length of both cylinders is less than $\delta\left(|v'_1-v_1|+|v_1 - v_2|\right)$, thus they can overlap only if $\theta$, the angle between $x_1 - x_2$ and $v_1' - v_1$, is small enough. 

- Suppose first 
that the lines  $x_1+\lambda (v'_1 -v_1)$ and $x_2+\mu( v'_1 - v_2)$ intersect at some point $u$ (see Figure \ref{fig: example}).
Then the length $\ell = \min \{ |u-x_1|, |u-x_2| \}$ satisfies 
$$
\ell = \frac{ |\sin \theta| }{ |\sin \alpha| } |x_1 - x_2| \, .
$$
For the intersection to occur one needs that $\ell \leq \delta\left(|v'_1-v_1|+|v_1 - v_2|\right)$ so that we get the condition on $\theta$
$$
|\sin \theta| \leq \delta\left(|v'_1-v_1|+|v_1 - v_2|\right) \frac{|\sin \alpha|}{|x_1 - x_2|} \leq  C\, \frac{|v'_1-v_1|+|v_1 - v_2|}{| \log \eps |^{1/2}}  \,
\cdotp
$$

- 
If the two lines in the picture do not intersect (as will happen in general for $d>2$), the above inequality can be proved
by a similar argument. Define $u, v$
as the points in the first and second lines where
the distance $2\eps$ between both lines is reached. 
Then we can project all vectors orthogonally to $u-v$, and we get exactly the same picture.

As a conclusion, we get that $\theta$ should belong to a solid angle of order 
$ \left(  \frac{|v'_1-v_1|+|v_1 - v_2|}{| \log \eps |^{1/2}} \right)^{d-1}$.
Integrating over $x_1'$ and $v_1'-v_1$, we deduce that 
\begin{equation*}
\eps^{1-d} \int dz'_1 \indc_{\{ |\sin \alpha| \leq \eps | \log \eps |^{1/2} \}}  |\zeta(z_1,z'_1) \zeta(z'_1,z_2)|
 e^{ - \frac{\beta}{4} \left(v_1^2+v_2^2 +|v'_1|^2\right)}   \leq  \frac{C}{| \log \eps |^{(1/2)(d-1)}}  \,
\cdotp
\end{equation*}
Combined with \eqref{eq: m=1 part1}, this completes \eqref{eq: m=1}.

\begin{figure}[h]
\begin{center}
\scalebox{0.4}{\includegraphics{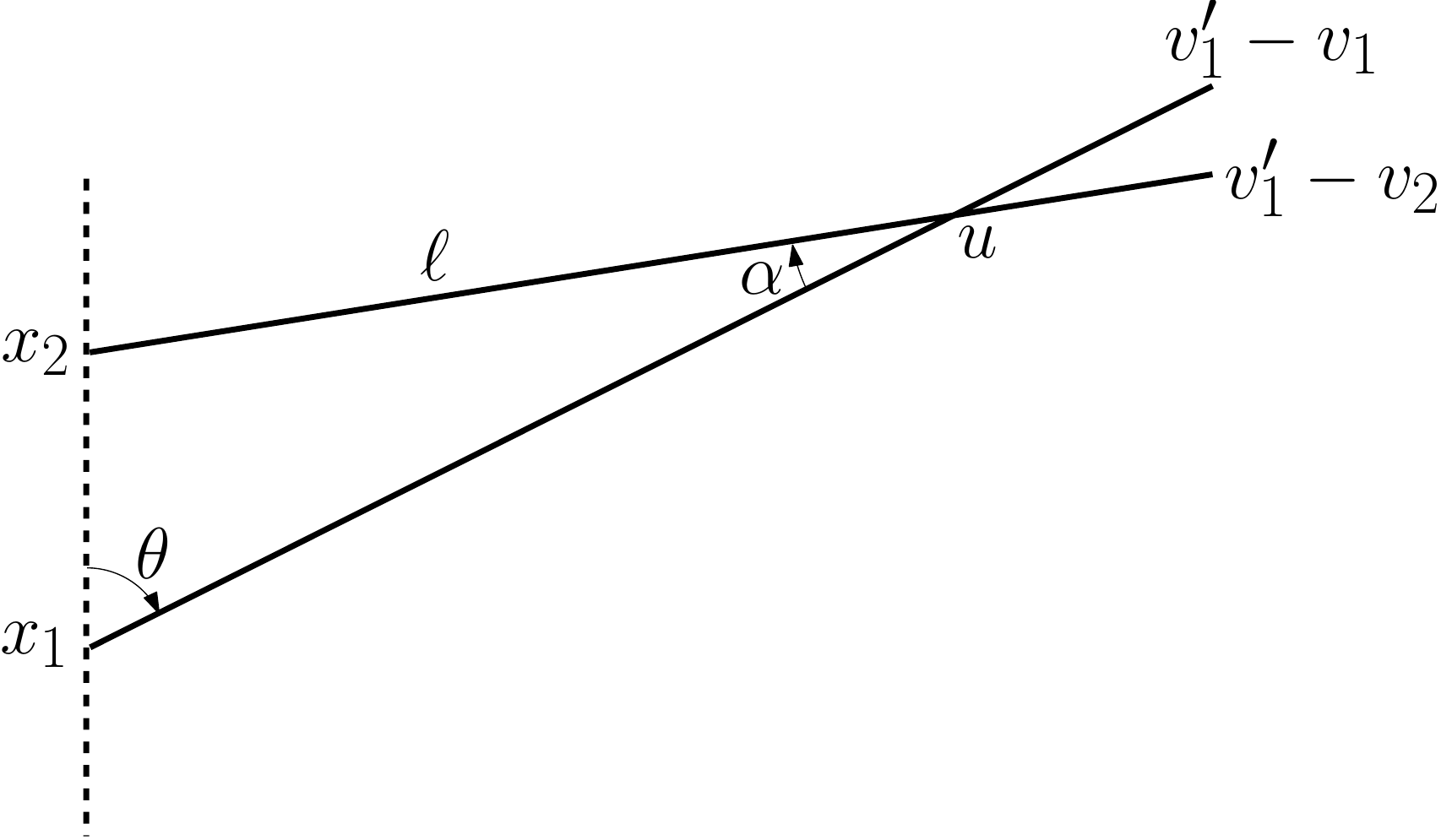}}
\caption{\label{fig: example} }
\end{center}
\end{figure}

\medskip

We now   show  that the contribution of the term $m$ is bounded by  
\begin{align}
\label{eq: decay m}
\eps^{m(1-d)} \int dZ'_m  |\zeta(z_1,z'_1)\zeta(z'_1,z'_2)\cdots \zeta(z'_m,z_2)|
e^{ - \frac{\beta}{4} \left(v_1^2 + v_2^2 + \sum_{i=1}^m |v'_i|^2\right)}  \leq \frac{ C^m \delta^{m-1}}{| \log \eps|^{1/2}}
\end{align}
for some constant $C$.
Summing over $m$ this will complete the derivation of \eqref{eq: decay} for $\delta$ small enough.

To estimate the case $m= 1$, we simply  used the fact that $|x_1 - x_2| \geq \eps  | \log \eps|$.
Suppose that $x_2'$ is such that $|x_1 - x_2'| \geq \eps  | \log \eps|$.
Then integrating with respect to  $z_1'$ leads to 
\begin{align*}
\eps^{m(1-d)} \int dZ'_m  &  \indc_{\{ |x_1 - x_2'| \geq \eps  | \log \eps| \}} \,   |\zeta(z_1,z'_1)\zeta(z'_1,z'_2)\cdots \zeta(z'_m,z_2)|
e^{ - \frac{\beta}{4} \left(v_1^2 + v_2^2 + \sum_{i=1}^m |v'_i|^2\right)} \\
& \leq 
\frac{C}{| \log \eps|^{1/2}}
\eps^{(m-1)(1-d)} \int dz'_2 \dots dz'_m  |\zeta(z'_2,z'_3)\cdots \zeta(z'_m,z_2)| e^{ - \frac{\beta}{8} \left(v_2^2 + \sum_{i=2}^m |v'_i|^2\right)}\;,
\end{align*}
where we applied an estimate as \eqref{eq: m=1} using part of the exponential factor, and removed the constraint $\indc$ in the upper bound. 
Finally, we can integrate term by term as the constraint on $z'_i$ depends only on $z'_{i+1}$. This leads to a contribution of the form $C\eps^{d-1} \delta ( |v'_i| + |v'_{i+1}|)$ for each constraint. After integrating the velocities, we obtain an upper bound $C^{m-1} \delta^{m-1} \eps^{(m-1)(d-1)} (1+|v_2|)e^{ - \frac{\beta}{8} v_2^2 }$ which implies an estimate as in \eqref{eq: decay m}.

It remains to consider  the set $\{  |x_1 - x_2'| \leq \eps  | \log \eps| \}$. We first integrate over $z'_2$
\begin{align*}
\eps^{m(1-d)} \int dZ'_m  &  \indc_{\{  |x_1 - x_2'| \leq \eps  | \log \eps| \}} \,   |\zeta(z_1,z'_1)\zeta(z'_1,z'_2)\cdots \zeta(z'_m,z_2)|
e^{ - \frac{\beta}{4} \left(v_1^2 + v_2^2 + \sum_{i=1}^m |v'_i|^2\right)} \\
& \leq 
\eps^{m (1-d)} C\eps^d  | \log \eps|^{d} e^{ - \frac{\beta}{4} \left(v_1^2 + v_2^2 \right)}
\int dz'_1  |\zeta(z_1,z'_1)| e^{ - \frac{\beta}{4} |v'_1|^2} \\
& \qquad \qquad \times 
\int dz'_3 \dots dz'_m  |\zeta(z'_3,z'_4)\cdots \zeta(z'_m,z_2)| e^{ - \frac{\beta}{4} \sum_{i=3}^m |v'_i|^2} .
\end{align*}
This breaks the cluster into two independent parts which can be estimated separately by the product of the volume of the cylinders,
leading to a higher order contribution $\eps  | \log \eps|^{d} C^m \delta^{m-1}$.
This completes the derivation of \eqref{eq: decay m} and the proof of \eqref{eq: factorisation hors cC} for $j=2$.

\end{proof}

The statement for $j=1$ is also similar and follows from the cluster expansion of \cite{Poghosyan Ueltschi}. In fact
Theorem 2.2 and Proposition~6.1 in \cite{Poghosyan Ueltschi} imply that $f_{\eps,0}^{(1)}$ is uniformly bounded 
by a geometric series for $\delta$ small.

Note that, in particular, the scaling condition \eqref{eq: Boltzmann Grad} holds, with $\kappa$ uniformly bounded in $\delta$.
Indeed, since there exists a (nontrivial) measurable nonnegative $f^{(1)}$ such that $ f_{\eps,0}^{(1)} \to f^{(1)} $ as $\e \to 0$,
it follows that 
$$\e^{d-1} {\mathbb E}_{\mu_\eps} \big( \cN \big) = \e^{d-1} \int \rho_{\eps,0}^{(1)}(z) dz = \mu_{\eps}^{-1}\int \rho_{\eps,0}^{(1)}(z) dz
= \int f_{\eps,0}^{(1)}(z) dz \to \kappa$$ where $\displaystyle\kappa := \int f^{(1)}(z) dz$.

\bigskip

The case $j>2$ can be treated similarly, however the expressions are more lengthy and we refer to \cite{Ue04} for details.

\section{Concluding remarks}

\subsection{Some wrong ideas about irreversibility}

The previous analysis  brings  a more precise understanding of Loschmidt's paradox~: it indicates where the   irreversibility of the Boltzmann description appears in the limiting process.

We would like first to comment upon some of the possible explanations which can be found in the literature.

\begin{itemize}

\item The direction of time in the Boltzmann dynamics is {\bf not related to an arbitrary choice 
 in writing the collision operator}. Once the initial data is prescribed, one has no choice in expressing the collision operator in terms of pre-collisional configurations for positive times, and in terms of post-collisional configurations for negative times. As explained in Remark \ref{collision-rmk}, this is the only way to define properly the traces by using the transport operator. This is also related to the fact that only the distribution of ingoing configurations has to be prescribed for the transport equation (see Remark \ref{boundary-rmk}).

\item  Irreversibility is not {\bf a direct consequence of chaos}. One can indeed start from a non chaotic initial data, in which case the Boltzmann hierarchy does not decouple.
However, even in this case, we have seen in Section \ref{irreversibility-section} that the limiting evolution is irreversible. We indeed have a Lyapunov functional, obtained by linear superposition of the entropy functionals with the Hewitt-Savage measure, which is strictly decreasing. 

\item 
Irreversibility is not due to {\bf neglecting the interaction length in the collision process}. In the limit, we forget indeed about the relative (microscopic) positions of the particles at the time of collisions, but this information could be kept by introducing an intermediate description, i.e. a simple modification of the Boltzmann equation referred to as the Enskog hierarchy \cite{PS}. In this equation the collision operator is still of type (2.3). However, Arkeryd and Cercignani \cite{AS} (see also~\cite{Bel}) prove that the Enskog equation (and thus the Enskog hierarchy using the previous superposition principle) is irreversible.

\end{itemize}

\subsection{A very singular averaging process}
Neglecting  spatial micro-translations in the limit induces   a first loss of information. 
The second loss of information, which is actually responsible for the loss of reversibility, consists in neglecting pathological configurations, i.e. configurations leading to pseudo-trajectories involving recollisions.
These sets~$\cB_{\eps_0}^\pm$ defined in~(\ref{eq: bad set})
  have a simple geometric definition, and their measure converges to 0 in the limit. So apparently it seems rather natural not to care about them.

The point is that the marginals at time $t$ can be computed as weighted averages of the initial marginals on very singular sets, which have exactly the same structure and the same measure. 
Recollisions of the backward dynamics are indeed exactly collisions of the forward pseudo-dynamics. We have therefore identified very precisely why time-concatenation is possible while reversing the arrow of time is not.
This can be summarized as in Figure~\ref{figuresingularsets}.

\begin{figure}[h]
\centering
\scalebox{0.4}{\includegraphics{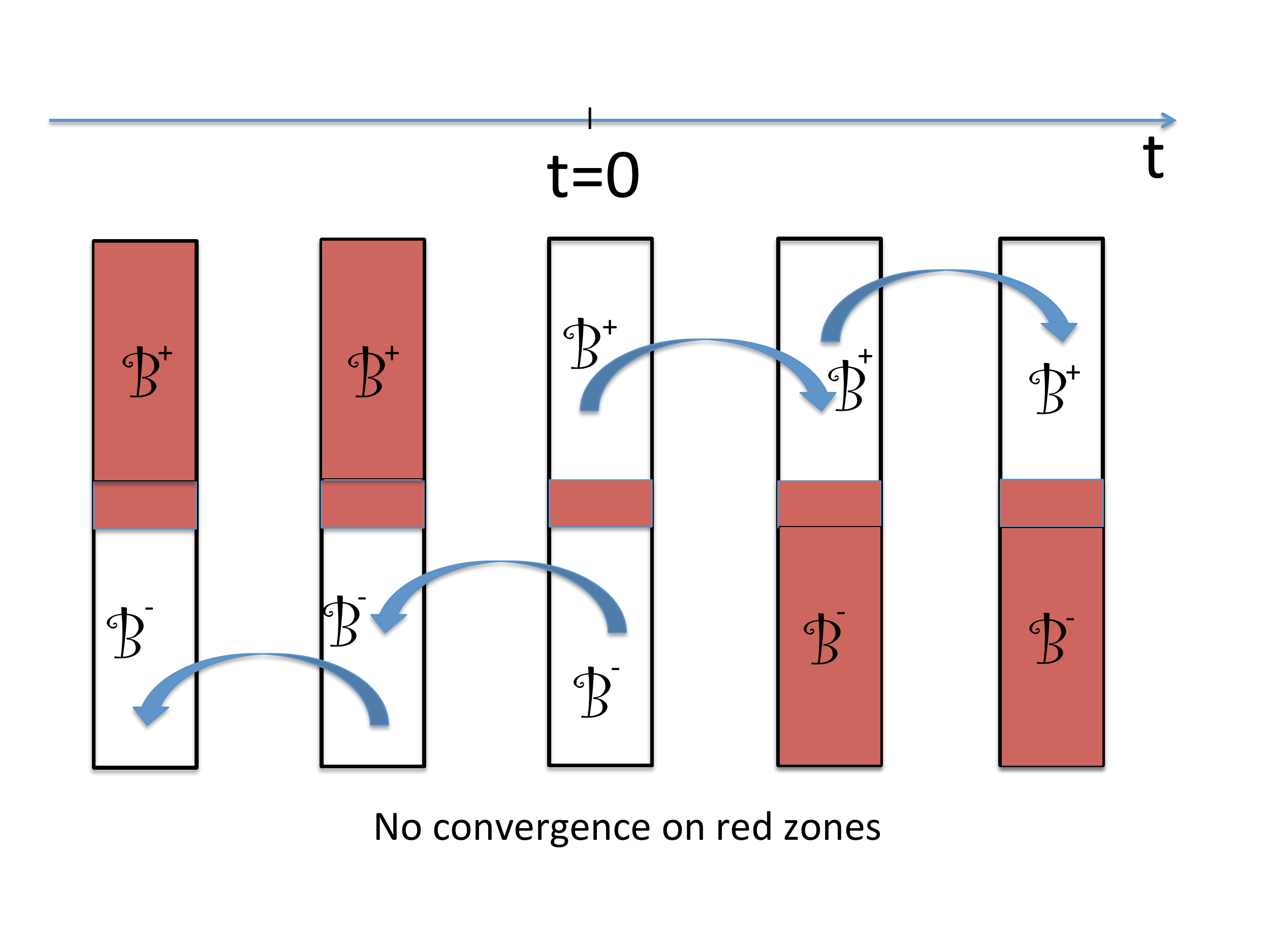}}
\caption{Convergence and lack of convergence over singular sets}
\label{figuresingularsets}
\end{figure}

Note that, 
for a better understanding of the Boltzmann dynamics,
it is not enough to look at the specific initial data (\ref{eq: initial data 2}), 
as its particular form is not stable under the dynamics. 
We would need a more systematic classification of the limiting dynamics 
depending on the microscopic structure of the $n$-particle distribution.

\end{document}